\RequirePackage[l2tabu,orthodox]{nag}

\documentclass[12pt,preprint,twoside,number,sort&compress]{elsarticle-modified}

\usepackage[english]{babel}
\usepackage[stretch=10]{microtype}
\usepackage[normalem]{ulem}
\usepackage[T1]{fontenc}
\usepackage[osf]{mathpazo}
\usepackage{inconsolata}
\usepackage[a4paper,left=2.5cm,right=2.5cm,top=4cm,bottom=2.5cm]{geometry}
\usepackage{color}
\usepackage{fancyhdr}
\usepackage[title]{appendix}
\usepackage[colorlinks,linkcolor=blue,bookmarksnumbered,pdfauthor={Xiangsheng Wang},pdftitle={Elliptic boundary value problem on non-compact G-manifolds}]{hyperref}
\usepackage{mathtools}
\usepackage{amsthm}
\usepackage{amssymb}
\usepackage{mathrsfs}

\biboptions{square}

\numberwithin{equation}{section}
\newtheorem{theorem}{Theorem}[section]
\newtheorem{lemma}[theorem]{Lemma}
\newtheorem{corollary}[theorem]{Corollary}
\newtheorem{definition}[theorem]{Definition}
\newtheorem{proposition}[theorem]{Proposition}

\theoremstyle{definition}
\newtheorem{remark}[theorem]{Remark}
\newtheorem{assumption}[theorem]{Assumption}
\newtheorem{problem}[theorem]{Boundary Value Problem}

\newcommand{\lno}{\Vert} 
\newcommand{\rno}{\Vert} 
\newcommand{\symup}{\mathrm}
\newcommand{\symbf}{\mathbf}
\newcommand{\diff}{\mathop{}\!\symup{d}}

\DeclareMathOperator{\ind}{Ind} 
\DeclareMathOperator{\dom}{dom}
\DeclareMathOperator{\coker}{coker}
\DeclareMathOperator{\ran}{ran}

\DeclareMathOperator{\supp}{supp}

\begin{document}
    
\begin{frontmatter}

\title{Elliptic boundary value problem on\\ non-compact $G$-manifolds} 
\author{Xiangsheng Wang}
\ead{wangxs1989@gmail.com}
\address{Chern Institute of Mathematics \& LPMC, Nankai University, Tianjin, P.R. China 300071}

\begin{abstract}
In this paper, an equality between the Hochs-Mathai type index and the Atiyah-Patodi-Singer type index is established when the manifold and the group action are both non-compact, which generalizes a result of Ma and Zhang for compact group actions. As a technical preparation, a problem concerning the Fredholm property of the global elliptic boundary value problems of the Atiyah-Patodi-Singer type on a non-compact manifold is studied.
\end{abstract}

\begin{keyword}
    {\small Elliptic boundary value problem \sep Locally compact groups \sep Fredholm property \sep Atiyah-Patodi-Singer type index \sep Hochs-Mathai type index}
    
    \MSC[2010] {\small 58J32 (Primary) \sep 58F06 \sep 58J20 \sep 53D50 \sep 53C27 (Secondary)}
\end{keyword}

\end{frontmatter}

\section{Introduction} 
\label{sec:introducton}

\noindent Let $M$ be a compact Spin$^c$ manifold with nonempty boundary carrying a compact group action $G$. In \cite{Ma:2014ja}, to prove Vergne's conjecture on $[Q,R]=0$ principle, more explicitly, quantization commuting with symplectic reduction principle, for non-compact symplectic manifolds, Ma and Zhang established an important equality between the transversal index and the Atiyah-Patodi-Singer\footnote{We'll use the acronym APS in the sequel.} type index on $M$,
\begin{equation}
    \label{eq:mz}
    \text{\texttt{Transversal index}}=\text{\texttt{APS type index}}.
\end{equation}

To prove (\ref{eq:mz}), a key ingredient is to use an index introduced by \citet{Braverman:2002wy}, which coincides with the transversal index by a result of Braverman, c.f. \cite{Braverman:2002wy}, \cite{Ma:2012tq}. \citet{Ma:2014ja} proved the following equality between two kinds of indices on $M$,
\begin{equation}
    \label{eq:mz2}
    \text{\texttt{Braverman index}} = \text{\texttt{APS type index}}.
\end{equation}

To consider problems concerning the $[Q,R]=0$ principle for the non-compact case, it is desirable to generalize the above kind of equality to the case that both the manifold and the acting group are non-compact. The main purpose of this paper is to establish such an equality. More precisely, let $M$ be an even dimensional non-compact oriented symplectic manifold equipped with a prequantum line bundle $(L,h^L,\nabla^L)$, on which a locally compact group $G$ acts symplectically. Suppose the group action is both proper and cocompact. Then we prove that \uline{an indices equality of type (\ref{eq:mz2})} still holds on $M$.
\footnote{Strictly speaking, the left hand side of (\ref{eq:mz2}) should be defined on a manifold without boundary containing $M$. We will describe the precise condition in Section \ref{sec:braverman}.}

The main difficulty for such a generalization is to find proper definitions for both sides of (\ref{eq:mz2}) in the non-compact setting. On the left hand side, an index introduced by \citet{Hochs:2015iu} is a natural candidate. Therefore, the main effort of this paper centers around the right hand side of (\ref{eq:mz2}), i.e., establishing the well-definedness of the APS type index for non-compact manifolds. Admittedly, this is more or less of technical flavor.

Indeed, let $(M,g^{\symup{T}M})$ be an even dimensional non-compact Spin$^c$ Riemannian manifold with boundary. To carry out the usual argument about $L^2$ sections, one relies on a cut-off function, say $f$, as in \citep{Mathai:2010bt}. Let $S(\symup{T}M)=S_{+}(\symup{T}M)\oplus S_{-}(\symup{T}M)$ be the spinor bundle and $E$ be the coefficient Hermitian vector bundle respectively. Suppose that $\Gamma(M, S(\symup{T}M)\otimes E)^G$ is the space of invariant smooth sections. An orthogonal projection operator $P_f$ is defined by $f$ on the space of $L^2$ sections of $S(\symup{T}M)\otimes E$, whose range, denoted by $\symbf{H}^0_f(M,S(\symup{T}M)\otimes E)^G$, is the closure of $f\Gamma(M, S(\symup{T}M)\otimes E)^G$ in the space of $L^2$ sections. In the same fashion, higher Sobolev spaces $\symbf{H}^i_f(M,S(\symup{T}M)\otimes E)^G$ associated to $f$ are also defined. 

Roughly, two aspects should be checked when one tries to adapt the definition of the APS type index in order to take $f$ into consideration. Firstly, it's well-known that the APS type index involves a kind of global boundary condition, the so-called APS boundary condition, whose definition relies on the spectral projection of the boundary operator. In this paper, with the help of $f$, certain spectral projection, denoted by $P_{\ge 0, f}$, plays a similar role on a non-compact manifold. Secondly, the Fredholm property of the Dirac operator with the APS boundary condition should be reestablished. For this purpose, we will apply the elementary functional analysis method as in \citep{Atiyah:1975uz} and \citep{Bar:2012in}. More concretely, our strategy is as follows. A partial result, the Fredholmness of the boundary value problem with the product structure assumption, is proved in Subsection \ref{sub:product_case}. Then the result is extended to the general case in Subsection \ref{sub:general}, using the Rellich perturbation theorem about self-adjoint operators. 

Put it briefly, the major technical problem is the solution of the following boundary value problem on a non-compact manifold.
\begin{problem}
	\label{prob:1}
	Let $D$ be the Spin$^c$ Dirac operator on $M$. Define the following boundary value problem,
	\begin{equation*}
		(P_fD_+,P_{\ge 0,f}):\dom P_fD_+ \rightarrow \mathbf{H}^0_f(M,S_-(\symup{T}M)\otimes E)^G ,
	\end{equation*}
	where
    \begin{equation}
        \label{eq:bvp1}
        \dom P_fD_+ =\{\phi \in \symbf{H}^1_f(M,S_+(\symup{T}M)\otimes E)^G|P_{\ge 0,f} (\phi|_{\partial M}) = 0\}. 
    \end{equation} 
    	
\noindent \textit{Question}. Is $(P_fD_+,P_{\ge 0,f})$ a Fredholm operator?
\end{problem}

The rest of the paper is organized as follows. The first two sections contain preliminary materials. Geometric background and notation conventions are summarized in Section \ref{sec:aps}. Section \ref{sec:dirac} contains some results on the property of the operator $P_fD$ on a manifold with or without boundary. With these preparations, Boundary Value Problem \ref{prob:1} will be solved in Section \ref{sec:product}, see Theorem \ref{thm:main1} for a precise statement. Section \ref{sec:braverman} is devoted to the generalization of Ma-Zhang's result (\ref{eq:mz2}), specifically, we formulate and prove it as Theorem \ref{thm:fin}.

\section{Notation and assumption} 
\label{sec:aps}

Let $M$ be an even dimensional non-compact oriented Spin$^c$ manifold with nonempty boundary $\partial M$ and $\dim{M}=n$. The orientation on $\partial M$ is induced by that of $M$. Moreover, assume that $M$ admits a left action by a locally compact Lie group $G$. About the group action, we make the following  assumption.
\begin{assumption}
	The group action is \emph{proper} and \emph{cocompact}.
\end{assumption}
The properness means that the following map
\begin{equation*}
	\begin{gathered}
		G\times M \rightarrow M\times M \\
		(g, x) \mapsto (x,gx)
	\end{gathered}
\end{equation*} 
is proper. And we say a group action is cocompact if and only if the quotient space $M/G$ is compact. Sometimes we also say a subset $U$ of a manifold with group action $G$ is cocompact, which means that $U/G$ is compact.

We will choose a fixed right invariant Haar measure $\symup{d}g$ on $G$ in the whole paper. A prominent feature of a proper action is that there exists a smooth, non-negative function $c(x)$ on $M$ whose support intersects every orbit of the action in a compact subset, c.f. \citep{Bourbaki:2004vy}. A function with such property will be called a \emph{cut-off function}. Furthermore, a cut-off function $c$ can be normalized as,
\begin{equation}
	\label{eq:nor}
	\int_G c^2(gx) \diff g =1,
\end{equation}
although such normalization is not always needed.

As a first application of cut-off functions, we construct a $G$-invariant Riemannian metric on $M$. Starting with an arbitrary metric on the tangent bundle $\symup{T}M$, one can obtain a $G$-invariant metric $\bar{g}$ by using the following averaging process,
\begin{equation*}
	\bar{g}(v,w)(x)= \int_G c^2(hx) g(h_*v,h_*w)(hx) \diff h.
\end{equation*}
From now on, we can and we will assume that $G$ acts on $(M,g^{\symup{T}M})$ isometrically.

Fix a Spin$^c$ structure over the Riemannian manifold $(M,g^{\symup{T}M})$. Since the metric is $G$-invariant, we make a further assumption that the $G$-action on the orthogonal frame bundle can be lifted up to the Spin$^c$ principal bundle. In other word, we suppose that the Spin$^c$ structure on $M$ is $G$-equivariant.\footnote{We take this assumption mainly for convenience. Instead one can also assume that the geometric data on spinors, more general on Clifford module, are $G$-invariant directly, just like \citep{Berline:2004vr}.} The spinor bundle associated to the Spin$^c$ structure is denoted by $S(\symup{T}M)=S_{+}(\symup{T}M)\oplus S_{-}(\symup{T}M)$. With the metric $g^{\symup{T}M}$ and Levi-Civita connection $\nabla^{\symup{T}M}$ on $TM$, one can define the metric $g^{S(\symup{T}M)}$ and connection $\nabla^{S(\symup{T}M)}$ on $S(\symup{T}M)$ as usual.\footnote{More precisely, to define metric and connection on spinor bundle, one needs the information on the determinant line bundle of the Spin$^c$ structure, c.f. \citep{LawsonJr:1989ub}.} Moreover, we introduce an auxiliary complex vector bundle $(E,h^E,\nabla^E)$ over $M$, which carries a $G$-invariant Hermitian metric $h^E$ and a $G$-invariant Hermitian connection $\nabla^E$. Now, the metric $\left\langle \cdot, \cdot \right\rangle$ and connection $\nabla^{S(\symup{T}M)\otimes E}$ on $S(\symup{T}M)\otimes E$ follow from those of $S(\symup{T}M)$ and $E$, both of which are $G$-invariant. Whenever there is no ambiguity, instead of $\nabla^{S(\symup{T}M)\otimes E}$, a short symbol $\nabla$ will be used. Finally, one notices that as a Clifford module of $\symup{T}M$, the Clifford action on $S(\symup{T}M)\otimes E$ is also $G$-equivariant.

	For another application of cut-off functions, let us formulate a suitable variant of Sobolev spaces on non-compact manifolds with group action.  We begin with a more explicit construction of cut-off functions with the help of  cocompactness condition. Since $M/G$ is compact, a compact subset $Y$ of $M$ exists such that $G(Y)=M$, c.f. \citep{Phillips:1989wa}. Equivalently, we can say that $M$ is saturated by $Y$ in Bourbaki's language, c.f.  \citep{Bourbaki:2004vy}. Let $U$ and $U'$ be two precompact open subsets of $M$, which satisfy the following inclusion relations,
	\begin{equation}
        \label{eq:setr}
		Y \subseteq U \subseteq \overline{U} \subseteq U'.
	\end{equation}
    Now one can find a function $f:M\rightarrow [0,1]$ such that $f|_U=1$ and $\supp (f)\subseteq U'$. Clearly, $f$ satisfies the definition of a cut-off function (without normalization). Throughout this paper, we will use a fixed $f$ as our cut-off function. Moreover, let $\chi$ be the normalization of $f$, that is,
    \begin{equation}
        \label{eq:xd}
    	\chi(x) = \frac{f(x)}{(\int_G f^2(gx) \diff g)^{1/2}}.
    \end{equation}
	
    As usual, we identify the cotangent bundle $\symup{T}^*M$ with the tangent bundle $\symup{T}M$ via the Riemannian metric. Especially, $\symup{T}^*M$ will have a metric induced from the metric on $\symup{T}M$, as well as a connection $\nabla^{\symup{T}^*M}$. Using the tensor product construction of connections, $\nabla^{S(\symup{T}M)\otimes E}$ and $\nabla^{\symup{T}^*M}$ can be used to define a connection on $(\bigotimes^i \symup{T}^*M)\otimes S(\symup{T}M)\otimes E$, which maps from $(\bigotimes^i \symup{T}^*M)\otimes S(\symup{T}M)\otimes E$ to $(\bigotimes^{i+1} \symup{T}^*M)\otimes S(\symup{T}M)\otimes E$. Iterating such construction $i$-th times, the symbol $\nabla^i$ denotes the resulting operator, 
    \begin{equation*}
        \nabla^i:\Gamma(M,S(\symup{T}M)\otimes E)\rightarrow \Gamma(M,(\otimes^i \symup{T}^*M)\otimes S(\symup{T}M)\otimes E).
    \end{equation*}
    Note $\nabla=\nabla^1=\nabla^{S(\symup{T}M)\otimes E}$, which is in conformity with our symbol conventions. With the obvious metric on the tensor product bundle, we define the usual Sobolev $k$-norm of a smooth section $\phi\in \Gamma(M,S(\symup{T}M)\otimes E)$ like in \citep[Ch. III, (2.1)]{LawsonJr:1989ub},    
	\begin{equation}
        \label{eq:sn}
        \Vert \phi \Vert^2_{k,M} =\sum^{k}_{i=0} \int_M \Vert \nabla^i \phi \Vert^2_{(\otimes^i \symup{T}^*M)\otimes S(\symup{T}M)\otimes E} \diff v,
	\end{equation}
	where $\symup{d}v$ is the Riemannian volume element on $(M,g^{\symup{T}M})$. The completions of smooth sections of $S(\symup{T}M)\otimes E$ with these norms are denoted by $\symbf{H}^k(M, S(\symup{T}M)\otimes E)$. Obviously, $\symbf{H}^0$ is just the space of $L^2$ sections of $S(\symup{T}M)\otimes E$, we will use two terms interchangeably. If there is no ambiguity, we often omit the subscript indicating the manifold on which the Sobolev norms are calculated.
    
    In appearance of a $G$-action, we mainly concern the $G$-invariant sections of $S(\symup{T}M)\otimes E$. Let $\Gamma(M,S(\symup{T}M)\otimes E)^G$ be the subspace that consists of $G$-invariant smooth sections. The completions of $f\Gamma(M,S(\symup{T}M)\otimes E)^G$ under the $k$-th Sobolev norms are denoted by $\symbf{H}^k_f(M, S(\symup{T}M)\otimes E)^G$ respectively, which are spaces that we are working with in this paper. 
    
    In many ways, $\symbf{H}^k_f(M, S(\symup{T}M)\otimes E)^G$ behave like the usual Sobolev spaces. In fact, if the quotient space $M/G$ does own a differential structure, there exists a natural bijection, although not an isometry in general, between usual Sobolev spaces on $M/G$ and $\symbf{H}^k_f(M, S(\symup{T}M)\otimes E)^G$. A simple but useful feature of $\symbf{H}^0_f(M, S(\symup{T}M)\otimes E)^G$ is the following estimate, c.f. \citep[(2.6)]{Mathai:2010bt}. Using the same notations as in (\ref{eq:setr}), for $s\in \Gamma(M,S(\symup{T}M)\otimes E)^G$, one has
	\begin{equation}
		\label{eq:fest}
		\lno s \rno_{U,0} \le \lno fs \rno_0 \le \lno s \rno_{U',0} \le C\lno s \rno_{U,0},
	\end{equation}
    where $\lno \cdot \rno_{U,0}$ is the $L^2$-norm on $U$ and $C$ is a positive constant independent of $s$.
	
	Following \citet{Mathai:2010bt}, we define a particular projection operator associated to a cut-off function, which will play a key role in the whole paper.
	\begin{definition}
		Define $P_f$ to be the orthogonal projection operator, which maps from $\symbf{H}^0(M,\allowbreak S(\symup{T}M)\otimes E)$ onto $\symbf{H}^0_f(M,S(\symup{T}M)\otimes E)^G$.
	\end{definition}
	
    Quite satisfactory, we can write down an explicit expression for $P_f$. Define the modular function for $G$ as $\delta:\, G \rightarrow \mathbb{R}^+$ satisfying $\symup{d}(g^{-1})=\delta(g)\diff g$. Bunke \citep[Appendix]{Mathai:2010bt}, in the unimodular case, and Tang, Yao and Zhang \citep[Proposition 3.1]{Tang:2013dq}, in the general case, proved that, for $\phi \in \symbf{H}^0(M,S(\symup{T}M)\otimes E)$,
	\begin{equation}
		\label{eq:proj}
        (P_f \phi)(x)=\frac{f(x)}{A(x)^2} \int_G \delta(g)f(gx)g^{-1}(\phi(gx))\diff g,
	\end{equation}
	where 
	\begin{equation*}
		A(x)=(\int_G \delta(g)f^2(gx)\diff g)^{1/2}.
	\end{equation*}
	By (\ref{eq:proj}), $P_f$ maps smooth sections to smooth sections.
	
    In order to define the APS boundary condition, we review the definition of Spin$^c$ Dirac operators briefly. Let $e_1,e_2,\dots,e_n$ be a locally oriented orthonormal frame of $\symup{T}M$. Recall that $\nabla$ is the Clifford connection, i.e., both unitary and compatible with the Clifford action. One has the following standard definition of the Spin$^c$ Dirac operator on $M$, c.f. \citep[Appendix D]{LawsonJr:1989ub},
	\begin{equation}
		\label{eq:dd}
		D = \begin{bmatrix} 0& D_-\\ D_+& 0 \end{bmatrix} = \sum^{n}_{i=1} c(e_i)\nabla_{e_i}:\Gamma(M,S(\symup{T}M)\otimes E) \rightarrow \Gamma(M,S(\symup{T}M)\otimes E),
	\end{equation}
    where $c(\cdot)$ denotes the Clifford action. Due to the $G$-equivariance of the Clifford action and connection, $D$ is a $G$-equivariant differential operator. As a result, the space of invariant sections $\Gamma(M,S(\symup{T}M)\otimes E)^G$ is preserved by $D$.
	
    On the boundary, one can define another Dirac operator $D_{\partial M}$. Let $e_1,e_2,\dots,e_n$ be a locally oriented orthonormal frame around a boundary point $p$. Besides, we require $e_n$ to be the inward unit normal vector field on the boundary. Following \citep[Lemma 2.2]{Gilkey:1993dm}, $D_{\partial M}:\Gamma(\partial M,(S(\symup{T}M)\otimes E)|_{\partial M}) \rightarrow \Gamma(\partial M,(S(\symup{T}M)\otimes E)|_{\partial M})$, can be written as
	\begin{equation}
		\label{eq:pdd}
		D_{\partial M} = \begin{bmatrix} D_{\partial M,+} &0\\0& D_{\partial M,-} \end{bmatrix} = - \sum^{n-1}_{i=1}c(e_n)c(e_i)\nabla_{e_i} +  \frac{1}{2}\sum^{n-1}_{i=1}\pi_{ii},
	\end{equation}
	where
	\begin{equation*}
		\pi_{ij}=\left\langle \nabla^{TM}_{e_i}e_j,e_n \right\rangle,\;1\le i,j \le n-1,
	\end{equation*}
    is the second fundamental form of $\partial M$. $D_{\partial M}$ is also a $G$-equivariant differential operator. But, unlike $D$, $D_{\partial M}$ preserves the $\mathbb{Z}_2$-grading of $(S(\symup{T}M)\otimes E)|_{\partial M}$. 
    
    By definition, the restricted function $f|_{\partial M}$ is a cut-off function of $\partial M$. Using (\ref{eq:proj}), one can find that for any $\phi\in \Gamma(M,S(\symup{T}M)\otimes E)$, $(P_f \phi)|_{\partial M} = P_{f|_{\partial M}}(\phi|_{\partial M})$. In case of no confusion, we still denote $P_{f|_{\partial M}}$, the projection operator on the boundary, by $P_f$. The boundary operators we are interested in are
	\begin{equation*}
		\label{eq:pbdo}
		P_fD_{\partial M,\pm}P_f: f\Gamma\big(\partial M,(S_{\pm}(TM)\otimes E)|_{\partial M}\big)^G \rightarrow f\Gamma\big(\partial M,(S_{\pm}(TM)\otimes E)|_{\partial M}\big)^G,
	\end{equation*}
	which satisfy the identities
	\begin{equation*}
		c(e_n)P_fD_{\partial M,\pm}P_f = -P_fD_{\partial M,\mp}P_fc(e_n).
	\end{equation*}
	
     We will prove in Section \ref{sec:dirac} that $P_fD_{\partial M,+}P_f$ (resp. $P_fD_{\partial M,-}P_f$) has a spectral decomposition with respect to $\symbf{H}^0_f(M, S_{+}(\symup{T}M)\otimes E)^G$ (resp. $\symbf{H}^0_f(M, S_{-}(\symup{T}M)\otimes E)^G$). With this result, the spectral projection operators appearing in the APS boundary condition are defined as follows.
     \begin{definition}
         \label{def:projaps}
          Define $P_{\ge 0,f}$ to be the orthogonal projection operator, which maps from $\symbf{H}^0_f(\partial M,\allowbreak (S_{+}(\symup{T}M)\otimes E)|_{\partial M})^G$ onto the closed subspace spanned by eigenspaces of $P_fD_{\partial M,+}P_f$ associated with nonnegative eigenvalues. Set $P_{< 0,f}=1-P_{\ge 0,f}$.
     \end{definition} 
     Note that both of the spectral projections are defined by $P_fD_{\partial M,+}P_f$. Throughout this paper, we stick to the ``plus'' version. Nevertheless, one can certainly define the ``minus'' counterpart using $P_fD_{\partial M,-}P_f$.
     
     With $P_{\ge 0,f}$, one can formulate the APS boundary condition.
     \begin{definition}
         \label{def:aps}
         For any section $s\in \Gamma(M,S_+(\symup{T}M)\otimes E)^G$, the APS boundary condition is defined to be $P_{\ge 0, f}((fs)|_{\partial M}) = 0$.
     \end{definition}
     With this definition, we can say that the boundary condition (\ref{eq:bvp1}) used in Boundary Value Problem \ref{prob:1} is exactly the APS boundary condition. In Section \ref{sec:product}, we will prove that the operator $(P_fD_{+},P_{\ge 0,f})$ in Boundary Value Problem \ref{prob:1} has a finite index. 
     \begin{definition}
         \label{def:apsind}
         The APS type index is defined to be the index of operator $(P_fD_{+},P_{\ge 0,f})$.
     \end{definition}
	
    Finally, let us comment a little on the geometric structure near the boundary. For a compact Riemannian manifold $Z$ with boundary $\partial Z$, it's a basic fact that for small $\epsilon$, the manifold near the boundary is diffeomorphic to $\partial Z\times [0,\epsilon]$ by exponential map. For simplicity, we always assume $\epsilon=1$. In our non-compact settings, since the group action on $M$ is cocompact and the Riemannian metric of $M$ is $G$-invariant, the same conclusion also holds. Moreover, the aforementioned diffeomorphism should be $G$-equivariant. If the diffeomorphism is also an isometry, i.e. $\partial M\times [0,1]$ has product metric, and the metric and connection on $E|_{\partial M\times [0,1]}$ are constant in the normal direction, we say that $M$ (with vector bundle $E$) has a (metric) \uline{product structure} near the boundary. Clearly, when $M$ has a product structure, the second fundamental form terms in (\ref{eq:pdd}) vanish and $D_{\partial M}$ degenerates to the classical form used in \citep{Atiyah:1975uz}. This special case will also play a key role in this paper.


\section{Some analytic properties of $P_fD$} 
\label{sec:dirac}

In this section, we will collect miscellaneous analytic facts about $P_fD$ on $M$ according to whether $M$ has a boundary or not. Some of them have been known, c.f. \citep{Mathai:2010bt}, but we summarize them here for the sake of completeness. In particular, we incorporate a proof on the existence of spectral decomposition of $P_fD_{\partial M}$,\footnote{To be more consistent with notations in Section \ref{sec:aps}, we should use $P_fD_{\partial M}P_f$ here. But since two operators coincide when acting on $f \Gamma(M,(S(\symup{T}M)\otimes E)|_{\partial M})$, we use this shorter symbol from time to time.} which, as we have indicated in Section \ref{sec:aps}, is essential for the definition of the APS boundary condition.

\subsection{The case of $M$ without boundary} 
\label{sub:empty}

In this subsection, we assume $\partial M = \emptyset$ temporarily. However, the requirement on the dimension of $M$ can be relaxed. The argument works for both odd and even dimensional manifolds. Apart from these two points, all other assumptions and notations will be the same as in Section \ref{sec:aps}. Basically, we are going to reprove some results for $P_fD$ that is well-known for self-adjoint elliptic operators on compact manifold.

\begin{proposition}
	\label{prop:sd}
	Recall that $D$ is the Spin$^c$ Dirac operator and $P_f$ is the projection operator defined by a cut-off function $f$. About $P_fD$, the following facts hold,
	\begin{enumerate}
		\item Viewed as an unbounded operator acting on $\symbf{H}^0_f(M,S(\symup{T}M)\otimes E)^G$, $P_fD$ is essentially self-adjoint;
		\item The spectrum of $P_fD$ is a discrete subset of the real line, which consists only of eigenvalues, and all eigenspaces of $P_fD$ have finite dimension;
		\item $\symbf{H}^0_f(M,S(\symup{T}M)\otimes E)^G$ has an orthogonal direct sum decomposition whose summands are eigenspaces of $P_fD$, that is,
		\begin{equation*}
			\symbf{H}^0_f(M, S(\symup{T}M)\otimes E)^G=\bigoplus_{\lambda \in \symup{sp}(P_fD)} H_\lambda,
		\end{equation*} 
		where $H_\lambda$ is the eigenspace associated with an eigenvalue $\lambda$;
		\item All eigensections are smooth.
	\end{enumerate}
\end{proposition}

	Basically, one can prove Proposition \ref{prop:sd} in a familiar way as the elliptic operator on the compact manifold. The only ingredient that deserves further explanation seems to be a regularity result on the generalized solution of $P_fD$. We state it in the following lemma, whose proof is left in Appendix \ref{sec:simple}.
    
	\begin{lemma}
		\label{lemma:reg}
		As an unbounded operator acting on $\symbf{H}^0_f(M,S(\symup{T}M)\otimes E)^G$, the kernel of the adjoint operator of $P_fD$, i.e., the cokernel of $P_fD$,  consists of smooth sections.
	\end{lemma}
    
    With Lemma \ref{lemma:reg}, to prove Proposition \ref{prop:sd}, one can proceed in a standard way, c.f. \citep[Ch. III]{LawsonJr:1989ub}.
    
    \begin{proof}
        To begin with, we show the existence of the renowned Green's operator, which is a self-adjoint compact operator. At first, we notice that, by a formula in \citep[(2.13)]{Mathai:2010bt}, for $s\in \Gamma(M,S(\symup{T}M)\otimes E)^G$, one has
	\begin{equation}
		\label{eq:lme1}
		\lno P_fD(fs) \rno_0 \ge C_1\lno fs \rno_1 - C_2 \lno fs \rno_0,
	\end{equation}
	where $C_1$, $C_2$ are positive constants. By (\ref{eq:lme1}), one can verify that the closure of the unbounded operator $P_fD$ with domain $f \Gamma(M,S(\symup{T}M)\otimes E)^G$ is $P_fD$, whose domain is $\symbf{H}^1_f(M,S(\symup{T}M)\otimes E)^G$. In the following, $P_fD$ will be understood as a closed operator. Moreover, (\ref{eq:lme1}) also suggests that $P_fD$ has a closed range.
    
	Now, it's trivial to verify that $P_fD$ is formally self-adjoint. As a result, the kernel of $P_fD$ must be contained in its cokernel, combined with Lemma \ref{lemma:reg}, which implies,
	\begin{equation}
		\label{eq:keck}
		\ker P_fD = \coker P_fD.
	\end{equation}    
	By (\ref{eq:keck}), the following map is an isomorphism indeed,
	\begin{equation}
		\label{eq:dmap}
		P_fD:\symbf{H}^1_f(M,S(\symup{T}M)\otimes E)^G \cap (\ker P_fD)^{\bot} \rightarrow \symbf{H}^0_f(M,S(\symup{T}M)\otimes E)^G \cap (\ker P_fD)^{\bot},
	\end{equation}
	where the orthogonal complement is taken in $ \symbf{H}^0_f(M,S(\symup{T}M)\otimes E)^G$. Green's operator $R$ is defined to be the inverse of map (\ref{eq:dmap}), with $R|_{\ker P_fD}=0$. By definition, the range of $R$ is contained in $\symbf{H}^1_f(M, S(\symup{T}M)\otimes E)^G$. Since $P_fD$ is formally self-adjoint, $R$ is a bounded self-adjoint operator of $\symbf{H}^0_f(M, S(\symup{T}M)\otimes E)^G$. Moreover, by Rellich's lemma \ref{prop:rellich}, $R$ will also be a compact operator. A little spectral theory shows that $R$ has discrete real spectrum, all of which are eigenvalues. Especially, $R$ has a spectral decomposition with respect to $\symbf{H}^0_f(M,S(\symup{T}M)\otimes E)^G$. Now we can use $R$ to show the conclusions of the theorem by using the standard arguments.
    
    \end{proof}
    
    If $M$ has a nonempty boundary, applying Proposition \ref{prop:sd} to operator $P_fD_{\partial M}$, one has the following simple observation: the spectral projection operator $P_{\ge 0, f}$ used in the APS boundary condition is \uline{well-defined}.
    

\subsection{The case of $M$ with or without boundary} 
\label{sub:the_bounardry_of_m_may_not_be_empty}

In this subsection, we discuss several results which hold whether $M$ has a boundary or not. All of the following results are about the boundedness of certain operators between different Sobolev spaces.

	\begin{proposition}
		\label{prop:tb}
		$P_fD(1-P_f)$ and $(1-P_f)DP_f$ are bounded operators on $\symbf{H}^0(M,S(\symup{T}M)\otimes E)$.
	\end{proposition}
	\begin{proof}
    We first show that $P_fD(1-P_f)$ is the adjoint operator of $(1-P_f)DP_f$. Let $\sigma, \sigma'\in \Gamma(M,S(\symup{T}M)\otimes E)$. By Green's formula, c.f. \citep[Proposition 3.4]{BoossBavnbek:1993wb}, and the $G$-equivariance of Clifford action,
    \begin{multline}
        \label{eq:daj}
        (P_fD(1-P_f)\sigma,\sigma')=(D(1-P_f)\sigma,P_f \sigma')
        = \big((1-P_f)\sigma,DP_f \sigma'\big) \\ 
        + \int_{\partial M} \langle (1-P_{f|_{\partial M}})\sigma, P_{f|_{\partial M}}(c(e_n)\sigma') \rangle \diff v_{\partial M} = (\sigma, (1-P_f)DP_f \sigma'), 
    \end{multline}
    where in the last equality we use the orthogonality of $P_{f|_{\partial M}}$ on the boundary.
    
    By (\ref{eq:daj}), we only need to show the proposition for $(1-P_f)DP_f$. For $s\in \Gamma(M,S(\symup{T}M)\otimes E)^G$, by (\ref{eq:fest}), (\ref{eq:dd}), we have
	\begin{multline*}
		\|P_fD(fs)-D(fs)\|_0=\|(P_f-1)c(df)s\|_0\le \|c(df)s\|_0 \\
        \le C'_0\Vert s \Vert_{U',0} \le C'_1 \Vert s \Vert_{U,0} \le C_1 \|fs\|_0,
	\end{multline*}
	where $C'_0,C'_1,C_1$ are positive constants. The boundedness of $(1-P_f)DP_f$ follows from the above estimate.
	\end{proof}
    
    For the latter application, a result for $D^2$, in the same nature with the above proposition, is also displayed here.
    \begin{proposition}
        \label{prop:tb2}
        $P_fD^2-(P_fD)^2$ is a bounded operator on $\symbf{H}^0_f(M,S(\symup{T}M)\otimes E)^G$.
    \end{proposition}
    
    \begin{proof}        
        Let $s\in \Gamma(M,S(\symup{T}M)\otimes E)^G$. One has the following equality,
        \begin{equation}
            \label{eq:d21}
            \begin{aligned}                
                \big( P_fD^2-(P_fD)^2 \big)(fs) & = ( P_fD^2P_f - P_fDP_fDP_f)(fs) = P_fD(DP_f-P_fDP_f)(fs) \\
                & = P_fD(1-P_f)(DP_f(fs)) = \big(P_fD(1-P_f)\big)(c(df)s).
            \end{aligned}
        \end{equation}
        The result follows from (\ref{eq:d21}) and Proposition \ref{prop:tb}.
    \end{proof}

    Though $M$ is non-compact, thanks to the cut-off function $f$, for $s\in \Gamma(M, S(\symup{T}M)\otimes E)^G$, one has $\Vert P_fD(fs) \Vert_0 \le \Vert D(fs) \Vert_0 \le C \Vert fs \Vert_1$. Thus, $P_fD$ is a bounded operator, which maps from $\symbf{H}^1_f(M,S(\symup{T}M)\otimes E)^G$ to $\symbf{H}^0_f(M,S(\symup{T}M)\otimes E)^G$. The following obvious extension can be proved easily.
    
    \begin{proposition}
        \label{prop:bd}
        For $k\ge 1$, $P_fD$ is a bounded operator mapping from $\symbf{H}^k_f(M,S(\symup{T}M)\otimes E)^G$ to $\symbf{H}^{k-1}_f(M,S(\symup{T}M)\otimes E)^G$.
    \end{proposition}
    
    Finally, we recall the following result from \citep[Lemma 4.4]{Hochs:2015iu} and \citep[(2.27)]{Mathai:2010bt}.
    
    \begin{proposition}
        \label{prop:chi}
        Let $G$ be a unimodular group. For cut-off functions $\chi_1$ and $\chi_2$ satisfying normalization (\ref{eq:nor}), the Sobolev spaces $\symbf{H}^0_{\chi_i}(M,S(\symup{T}M)\otimes E)^G$, $i=1,2$, are naturally isometric. Moreover, for $s\in \Gamma(M,S(\symup{T}M)\otimes E)^G$, one has for $i=1,2$
        \begin{equation}
            \label{eq:xf}
            P_{\chi_i}D(\chi_i s)=\chi_i Ds.
        \end{equation}
    \end{proposition}



\section{A solution of Boundary Value Problem \ref{prob:1}} 
\label{sec:product}

In this section, we establish the well-definedness of the APS type index in the non-compact setting. Our main result is the following theorem, which answers the Boundary Value Problem \ref{prob:1} affirmatively.
\begin{theorem}
	\label{thm:main1}
    The operator,
	\begin{equation*}
		(P_fD_+,P_{\ge 0,f}):\dom P_fD_+ \rightarrow \symbf{H}^0_f(M,S_-(\symup{T}M)\otimes E)^G ,
	\end{equation*}
	is a Fredholm operator, whose domain is defined by
    \begin{equation}
        \label{eq:pbdc1}
        \dom P_fD_+ =\{\phi\in \symbf{H}^1_f(M,S_+(\symup{T}M)\otimes E)^G|P_{\ge 0,f} (\phi|_{\partial M}) = 0\},
    \end{equation} 
    where the spectral projection operator is defined through $P_fD_{\partial M,+}P_f$.
\end{theorem}
Before the proof, we discuss the trace theorem of Sobolev spaces $\symbf{H}^k_f(M,S(\symup{T}M)\otimes E)^G$ briefly, which is used implicitly in the formulation of the Boundary Value Problem \ref{prob:1}.
\begin{remark}
	\label{rk:fred}
    One basic result about the Sobolev spaces on a manifold with boundary is the trace theorem. According to this theorem, restricting to the boundary, referred to trace operator in the literature of PDE, will be a well-defined bounded operator from $\symbf{H}^s(M,S(\symup{T}M)\otimes E)$ to $\symbf{H}^{s-1/2}(\partial M,(S(\symup{T}M)\otimes E)|_{\partial M})$ if the manifold is compact. Although we did not define Sobolev spaces $\symbf{H}^k_f(M,S(\symup{T}M)\otimes E)^G$ when $k$ is a fraction, it's totally safe to talk about the bounded trace operators from $\symbf{H}^k_f(M,S(\symup{T}M)\otimes E)^G$ to $\symbf{H}^{k-1}_f(\partial{M},(S(\symup{T}M)\otimes E)|_{\partial M})^G$ for $k\ge 1$. This can be verified by the local essence of the trace theorem. As a result, for $\phi\in \symbf{H}^1_f(M,S(\symup{T}M)\otimes E)^G$, $\phi|_{\partial M}$ will lie in $\symbf{H}^0_f(\partial M,(S(\symup{T}M)\otimes E)|_{\partial M})^G$, which justifies the meaning of $P_{\ge 0,f}(\phi|_{\partial M})$ appearing in the Boundary Value Problem \ref{prob:1}.
\end{remark}

We will prove Theorem \ref{thm:main1} with the extra assumption that $M$ has product structure near the boundary in Subsection \ref{sub:product_case}. Result with full generality will be obtained in Subsection \ref{sub:general}. On technical level, the method we will use is rather close to that of \citet{Bar:2012in}.

\subsection{The product case} 
\label{sub:product_case}

	In this subsection, we assume that $M$ has product structure near the boundary. The proof splits into two steps. In the first step, we establish a G\"arding type inequality for the operator $P_fD_{+}$ with the APS boundary condition, which implies the closed range and finite-dimensional kernel property of $(P_fD_+,P_{\ge 0,f})$. In the second step, we identify the adjoint operator of $(P_fD_+,P_{\ge 0,f})$ explicitly, which in fact is $P_fD_{-}$ with a modified APS boundary condition. Combined with the result of the first step, one can establish the finite-dimensional cokernel property of $(P_fD_+,P_{\ge 0,f})$. We use the same notations as in Section \ref{sec:aps}.
	
	\vspace{2ex}
	\noindent \textbf{Step 1. (A G\"arding type inequality)} The object of this step is to establish the following inequality, which extends (\ref{eq:lme1}) to the case of manifolds with boundary. 
    \begin{lemma}
        \label{lemma:garding1}
        For $s\in \Gamma(M,S_+(\symup{T}M)\otimes E)^G$ and $fs$ satisfies the boundary condition (\ref{eq:pbdc1}), one has
    	\begin{equation}
    		\label{eq:gard}
    		\|P_fD_+(fs)\|_0\ge c_1\|fs\|_1 - c_2 \|fs\|_0.
    	\end{equation}
        where $c_1$, $c_2$ are positive constants independent of $s$.
    \end{lemma}
    
    \begin{proof}
    Firstly, by Proposition \ref{prop:tb}, for $s\in \Gamma(M,S_+(\symup{T}M)\otimes E)^G$, one has
	\begin{equation}
		\label{eq:pf}
		\|P_fD_+(fs)\|\ge \|D_+(fs)\|_0 - C_1 \|fs\|_0.
	\end{equation}
	Hence, it's enough to verify the G\"arding type inequality for $D_{+}(fs)$. Recall that $e_n$ is the inward unit normal vector field. By Green's formula, c.f. \citep[Proposition 3.4]{BoossBavnbek:1993wb},
	\begin{equation}
		\label{eq:ds}
		\|D_+(fs)\|^2_0=\int_M \langle fs, D^2(fs) \rangle \diff v_M + \int_{\partial M} \langle fs, c(e_n)D_+(fs) \rangle \diff v_{\partial M}.
	\end{equation}
	By the Lichnerowicz formula, c.f. \citep[Ch. II, Theorem 8.8 \& Appendix D, Theorem D.12]{LawsonJr:1989ub},
	\begin{equation}
		\label{eq:ds1}
		D^2 = -\Delta + \mathcal{O}(1),
	\end{equation}
	where $\Delta$ is the Bochner Laplacian. For the second term in (\ref{eq:ds}), by (\ref{eq:pdd}), the definition of $D_{\partial M}$, one has
	\begin{multline}
		\label{eq:ds2}
		\int_{\partial M} \langle fs, c(e_n)D_+(fs) \rangle \diff v_{\partial M} = \int_{\partial M} \langle fs, \nabla_{-e_n}(fs) \rangle \diff v_{\partial M} \\
        - \int_{\partial M} \langle fs, D_{\partial M,+} (fs) \rangle \diff v_{\partial M}.
	\end{multline}
	Combining (\ref{eq:ds})-(\ref{eq:ds2}) with the definition of the Bochner Laplacian, 
	\begin{multline}
		\label{eq:green}
		\|D_+(fs)\|^2_0 = \int_M \langle \nabla (fs), \nabla (fs)\rangle \diff v_M + \int_M \langle fs, \mathcal{O}(1)(fs) \rangle \diff v_M \\
        - \int_{\partial M} \langle fs, D_{\partial M,+} (fs) \rangle \diff v_{\partial M}.
	\end{multline}
    Since $fs$ satisfies the boundary condition (\ref{eq:pbdc1}), the spectral decomposition of $P_fD_{\partial M,+}$ gives
	\begin{equation}
		\label{eq:gn2}
		\int_{\partial M} \langle fs, D_{\partial M,+} (fs) \rangle \diff v_{\partial M} \le 0.
	\end{equation}
	As a result of (\ref{eq:green}) and (\ref{eq:gn2}), 
	\begin{equation}
		\label{eq:ds4}
		\|D_+(fs)\|^2_0 \ge \|fs\|_1^2 - C_2\|fs\|^2_0.
	\end{equation}
	By (\ref{eq:pf}), (\ref{eq:ds4}), we get (\ref{eq:gard}).        
    \end{proof}

    \begin{remark}
        \label{rk:gc}
        We notice that, in this step, the product structure assumption is actually not necessary. More precisely, the only difference between the general case and product case is that, in the general case, due to the non-vanishing second fundamental form, we will have an extra boundary term in (\ref{eq:ds2}). As a result, a new term that concerns about $\Vert fs \Vert^2_{0,\partial M}$, the $\symbf{H}^0$-norm of $fs$ on the boundary, will appear in (\ref{eq:ds4}). This is not a serious problem, since $\Vert fs \Vert^2_{0,\partial M}$ can be absorbed by $\|fs\|_1^2$, the $\symbf{H}^1$-norm of $fs$ in the interior, with an arbitrary small coefficient by the trace theorem and interpolation property of Sobolev spaces. 
    \end{remark}
	
	\vspace{2ex}
	\noindent \textbf{Step 2. (Existence of the adjoint operator)} In this step, we will show the existence of the adjoint operator of $(P_fD_+,P_{\ge 0,f})$. To state the result more clearly, we use the following ``adjoint version'' of Boundary Value Problem \ref{prob:1}. 
	\begin{problem}[Adjoint version]
		\label{probs:1}
		If $M$ is of product structure near the boundary, define the following boundary value problem,
		\begin{equation*}
			\label{eq:pbdc2}
			(P_fD_-,P_{< 0,f}):\dom{P_fD_-} \rightarrow \symbf{H}^0_f(M,S_+(\symup{T}M)\otimes E)^G,
		\end{equation*}
		where
        \begin{equation}
            \label{eq:bvp2}
            \dom P_fD_- =\{\phi\in \symbf{H}^1_f(M,S_-(\symup{T}M)\otimes E)^G|P_{< 0,f} c(-e_n)(\phi|_{\partial M}) = 0\}.
        \end{equation}
        
\noindent \textit{Question}. Is $(P_fD_-,P_{< 0,f})$ a Fredholm operator?
	\end{problem}
    
    \begin{remark}[The G\"arding type inequality for $P_fD_-$]
        \label{rk:gard}
        Clearly, the boundary condition (\ref{eq:bvp2}) is a modified APS boundary condition for operator $P_fD_-$. By a word-by-word translation of \textbf{Step 1}, we can show (\ref{eq:gard}) also holds for $(P_fD_-,P_{< 0,f})$.
    \end{remark}
    
    With these preparations, we will prove the following result.
	\begin{theorem}
		\label{thm:bd2}
		With the product structure assumption, the Boundary Value Problem \ref{probs:1} is the adjoint problem of the Boundary Value Problem \ref{prob:1} in the sense that $P_fD_-$ with domain $\dom{P_fD_-}$ is the adjoint operator of $P_fD_+$ with domain $\dom{P_fD_+}$ as unbounded operators on $\symbf{H}^0_f(M,S(\symup{T}M)\otimes E)^G$.
	\end{theorem}
    
    \begin{proof}
	To prove this theorem, the obstacle is to obtain suitable estimate near the boundary. With the product structure assumption, we follow closely as in \citep{Atiyah:1975uz} and \citep{Bar:2012in}, by using the $L^2$ expansion of sections near the boundary explicitly.
	
	To show Theorem \ref{thm:bd2}, we need to prove that,
    \begin{equation}
        \label{eq:dc}
        \dom P_fD_- = \dom\:\! (P_fD_+)^*,
    \end{equation}
    where $(P_fD_+)^*$ is the adjoint operator of $P_fD_+$ as an unbounded operator. 
    
    One side of (\ref{eq:dc}) is trivial, $\dom P_fD_- \subseteq \dom\: (P_fD_+)^*$, just by an application of Green's formula. The inclusion for the other direction is a regularity result essentially. The reason is that, by definition, the sections in $\dom\: (P_fD_+)^*$ only lie in $\symbf{H}^0$, while we need to show that they actually lie in $\dom P_fD_- \subseteq \symbf{H}^1$. By a simple argument using a partition of unity, we can split the regularity problem into two parts: in the interior and near the boundary. As to the interior regularity, there's no difference whether $M$ has a boundary or not. Hence, one can use the techniques for manifold without boundary as in Section \ref{sec:dirac}. More concretely, recall that we use elliptic regularity in Lemma \ref{lemma:reg} and Proposition \ref{prop:regv} in Appendix \ref{sec:simple} to show that the sections in cokernel of $P_fD$ are smooth. In the same fashion, if $\sigma \in \dom\: (P_fD_+)^*$ and $\supp(\sigma)\cap \partial M = \emptyset$, elliptic regularity implies $\sigma\in \symbf{H}^1$. Therefore, all the job left is to do some estimates near the boundary to establish the boundary regularity. 
	
    The trivial coordinates\footnote{In product case, trivial coordinate is the same as the geodesic coordinate.} $(y,u)$ on $\partial M \times [0,1]$ will be used, where $y$ represents the coordinates on $\partial M$ and $u$ is the coordinate in the normal direction. By the product structure assumption, one can suppose that every geometric object near the boundary is of product nature. Especially, we can and we will require the cut-off function $f$ to be independent of $u$. Then by (\ref{eq:proj}) and the $G$-equivariance of the Clifford action on $S(\symup{T}M)\otimes E$, one can verify the following commutation relations near the boundary,
    \begin{subequations}        
	\begin{gather}
		[P_f,c(e_n)]=0,\label{eq:comm1}\\
        [P_f,\partial_u]=0,\label{eq:comm2}
	\end{gather}
    \end{subequations}
    where $e_n$ is the inward unit normal vector field, while $\partial_u$ is the derivative with respect to the coordinate $u$. Note that (\ref{eq:comm1}) in fact has been used in (\ref{eq:daj}), though not written out explicitly. 
    
    Hence, near the boundary, $P_fD$ is independent of $u$. By the results of Section \ref{sec:dirac}, if $\phi\in \symbf{H}^0_f(M,S(\symup{T}M)\otimes E)^G$ and $\supp (\phi)\subseteq \partial M \times [0,1]$, one has the following Fourier expansion in the $L^2$ sense,
	\begin{equation}
		\label{eq:defp}
        \phi = \sum_{\lambda \in \symup{sp}(P_fD_{\partial M})} k_\lambda(u)\eta_\lambda(y),
	\end{equation}
    where $k_\lambda(u) \in L^2([0,1])$, $\eta_\lambda(y)$ is an eigensection of $P_fD_{\partial M}$ associated with eigenvalue $\lambda$.
	
    For every term in (\ref{eq:defp}), the action of $P_fD$ on it can be calculated explicitly. By (\ref{eq:comm1}), (\ref{eq:comm2}), for  $k_{\lambda}(u)\in \mathscr{C}^{\infty}_c([0,1))$, one has
	\begin{equation}
		\label{eq:cal}
		\begin{aligned}
            P_fD (k_{\lambda}(u)\eta_{\lambda}(y)) & =P_fc(e_n)(\partial_u +D_{\partial M}) (k_{\lambda}(u)\eta_{\lambda}(y)) \\
            & = c(e_n)P_f(\partial_u +D_{\partial M}) (k_{\lambda}(u)\eta_{\lambda}(y)) \\
            & = c(e_n)\big((\partial_u+\lambda)k_{\lambda}(u)\big)\eta_{\lambda}(y).
		\end{aligned}
	\end{equation}	
    
    Now, let $\sigma$ be a section in $\dom (P_fD_+)^*$ with compact support in $\partial M \times [0,1)$, by the definition of the adjoint operator, there exists a section $\tau \in \symbf{H}^0_f(M,S_+(\symup{T}M)\otimes E)^G$ such that
	\begin{equation}
		\label{eq:def}
        ( P_fD_+ \phi, \sigma ) = ( \phi, \tau )
	\end{equation}
    holds for any section $\phi \in \dom P_fD_+$, where $(\cdot,\cdot)$ is the inner product of $L^2$ sections. We will show $\sigma \in \dom P_fD_-$. Basically, we prove two things:
    \begin{enumerate}
        \item Every term appearing in the Fourier expansion of $\sigma$ lies in $\symbf{H}^1_f(M,S_-(\symup{T}M)\otimes E)^G$, of which the Fourier coefficient satisfies certain boundary condition;
        \item The Fourier expansion of $\sigma$ converges in $\symbf{H}^1$-norm, although it only converges in $\symbf{H}^0$-norm a priori.
    \end{enumerate}
        
    Let us write down the $L^2$ Fourier expansion of $-c(e_n)\sigma$ and $\tau$,
    \begin{subequations}        
	\begin{align}
        -c(e_n)\sigma &= \sum_{\lambda \in \symup{sp}(P_fD_{\partial M})} \sigma_{\lambda}(u)\eta_{\lambda}(y), \label{eq:sum}\\
		\tau &= \sum_{\lambda \in \symup{sp}(P_fD_{\partial M})} \tau_{\lambda}(u)\eta_{\lambda}(y). \label{eq:sum2}
	\end{align}
    \end{subequations}
	
    Let $\phi\in f\Gamma(M,S_+(\symup{T}M)\otimes E)^G$ be such a section that only one term $k_{\lambda}(u)\eta_{\lambda}(y)$ appears in its Fourier expansion. If $\phi\in \dom P_fD_+$, it must satisfy an extra constrain,
	\begin{equation}
		\label{eq:cond}
		k_{\lambda}(0)=0, \; \text{if }\lambda \ge 0.
	\end{equation}
    By (\ref{eq:cal}), (\ref{eq:def}), (\ref{eq:sum}) and (\ref{eq:sum2}), for and $k_{\lambda}\in \mathscr{C}^{\infty}_c([0,1))$ satisfying (\ref{eq:cond}), one has
	\begin{equation}
		\label{eq:key}
		\int^1_0 \big((\partial_u+\lambda)k_{\lambda}\big)\overline{{\sigma}_{\lambda}}\diff u = \int^1_0 k_{\lambda}\overline{\tau_{\lambda}}\diff u.
	\end{equation}
    Due to (\ref{eq:key}) and the arbitrariness of $k_{\lambda}$, by integration by part, one deduces that ${\sigma}_{\lambda}$ must be an absolutely continuous function and satisfy the additional boundary condition,
	\begin{equation}
		\label{eq:cond2}
		{\sigma}_{\lambda}(0)=0, \; \text{if }\lambda < 0.
	\end{equation}
    Meanwhile, for almost everywhere in $[0,1]$, one has
	\begin{equation}
		\label{eq:rs}
		(-\partial_u + \lambda){\sigma}_{\lambda} = \tau_{\lambda},
	\end{equation}
    which implies that every term in the Fourier expansion of $-c(e_n)\sigma$, i.e. $\sigma_{\lambda}(u)\eta_{\lambda}(y)$, lies in $\symbf{H}^1_f(M,S_-(\symup{T}M)\otimes E)^G$. Moreover, certain boundary condition (\ref{eq:cond2}) should be satisfied. Therefore, we complete our first goal.
	
	Next, we prove that the $L^2$ summation of (\ref{eq:sum}) converges in $\symbf{H}^1$-norm. Our tool is the G\"arding type inequality for $P_fD_-$. Put (\ref{eq:rs}) in another way,
	\begin{equation}
		\label{eq:rs2}
		\tau_{\lambda}(u)\eta_{\lambda}(y)=D_-( c(e_n){\sigma}_{\lambda}(u)\eta_{\lambda}(y)).
	\end{equation}
    By (\ref{eq:cond2}), for every $\lambda$,  $c(e_n){\sigma}_{\lambda}(u)\eta_{\lambda}(y)$ lies in $\dom P_fD_-$. By Remark \ref{rk:gard} and (\ref{eq:rs2}), we have the following estimate,
	\begin{equation}
		\label{eq:keye}
		\|{\sigma}_{\lambda}(u)\eta_{\lambda}(y)\|^2_1 \le C_5 \|c(e_n){\sigma}_{\lambda}(u)\eta_{\lambda}(y)\|^2_1 \le C_6 \|\tau_{\lambda}(u)\eta_{\lambda}(y)\|^2_0 + C_7 \|{\sigma}_{\lambda}(u)\eta_{\lambda}(y)\|^2_0.
	\end{equation}
    The estimate (\ref{eq:keye}) guarantees that the $L^2$ convergence of (\ref{eq:sum}) and (\ref{eq:sum2}) implies that (\ref{eq:sum}) is summable for $\symbf{H}^1$-norm. In other word, $\sigma\in \symbf{H}^1_f({M},S_-(\symup{T}M)\otimes E)^G$ has been proved. Besides, with the aid of the Fourier expansion of $\sigma$, one can check that $\sigma$ satisfies the boundary condition (\ref{eq:bvp2}) directly. 
    
    All in all, we have verified the domain relation (\ref{eq:dc}). Besides, by (\ref{eq:rs2}), we get $(P_fD_+)^*=P_fD_-$. The proof of Theorem \ref{thm:bd2} finishes.        
    \end{proof}
    
	With the preparation of the above two steps, we can prove Theorem \ref{thm:main1} in the product case readily.
	\begin{proof}[Proof of Theorem \ref{thm:main1} in the product case]
		By Lemma \ref{lemma:garding1}, the G\"arding type inequality, one sees that $(P_fD_+,P_{\ge 0,f})$ possesses a finite-dimensional kernel. Theorem \ref{thm:bd2} shows that the cokernel of $(P_fD_+,P_{\ge 0,f})$ is the kernel of $(P_fD_-,P_{<0,f})$ exactly. Hence, to show that $(P_fD_+,P_{\ge 0,f})$ has a finite-dimensional cokernel, it's enough to prove the corresponding property for the kernel of $(P_fD_-,P_{<0,f})$, which is the result of Remark \ref{rk:gard}, the G\"arding type inequality for $(P_fD_-,P_{<0,f})$. As to the closed ranged property of $(P_fD_+,P_{\ge 0,f})$, one can conclude it by using the G\"arding type inequality and
Peetre's lemma as in \citep[pp.~16, Lemma 2.1]{Kazdan:1983wy}.	
    \end{proof}


\subsection{The general case} 
\label{sub:general}
As we have indicated in Introduction, the product structure assumption for the Fredholmness result can be removed by a perturbation argument. Let us explain the idea more carefully. 

Intuitively, a general metric on $M$ can be approximated by the product metric, which gives rise to the possibility of approximating $P_fD$ by Fredholm operators. However, technically, we would like to apply the approximation method for $(P_fD)^2$ rather than $P_fD$. The principal reason is that unlike Boundary Value Problem \ref{prob:1}, the same type problem for the second order operator is a self-adjoint one. As a result, one can use the Rellich perturbation theorem to deduce the self-adjointness of $(P_fD)^2$, if the domain is chosen suitably, from the self-adjointness of the approximating operators. 

Thereupon, to carry out the perturbation argument, we need to generalize the result in Subsection \ref{sub:product_case} to a second order problem first, which involves several new calculations.
    \begin{theorem}
    	\label{thm:sq}
    	If $M$ is of product structure near the boundary, then the operator
    	\begin{equation*}
    		((P_fD)^2,\mathscr{B}_+):\mathscr{B}_+ \rightarrow \symbf{H}^0_f(M,S_+(\symup{T}M)\otimes E)^G
    	\end{equation*}
    	is a self-adjoint Fredholm operator, where
    	\begin{equation}
            \label{eq:sqbd}
            \mathscr{B}_+ =\{ \phi\in \symbf{H}^2_f(M,S_+(\symup{T}M)\otimes E)^G|P_{\ge 0,f}(\phi|_{\partial M})=P_{< 0,f}c(-e_n)\big((P_fD \phi)|_{\partial M}\big) =0 \}.
    	\end{equation}
        As usual, the spectral projection operators are defined by $P_fD_{\partial M,+}P_f$.
    \end{theorem}
    
    Clearly, Proposition \ref{prop:bd} testifies the well-posedness of Theorem \ref{thm:sq}. Our method to attack the second order problem is identical with the first order one, i.e. the G\"arding type inequality and the Fourier expansion on the cylindrical end. Thus we will organize the proof of Theorem \ref{thm:sq} in parallel with that of Theorem \ref{thm:main1} as much as possible.   

\vspace{2ex}    
\noindent \textbf{Step 1. (Another G\"arding type inequality)} As before, we establish a G\"arding type inequality for $((P_fD)^2,\mathscr{B}_+)$.

\begin{lemma}
    \label{lemma:garding2}
    For every $s\in \Gamma(M,S_{+}(\symup{T}M)\otimes E)^G$ with $fs$ satisfying the boundary condition (\ref{eq:sqbd}), one has
    \begin{equation}
        \label{eq:garding2}
            \lno (P_fD)^2(fs) \rno^2_0 \ge c'_1 \lno fs \rno^2_2 - c'_2 \lno fs \rno^2_0,
    \end{equation}
    where $c'_1$, $c'_2$ are positive constants independent of $s$.
\end{lemma}

\begin{remark}
    For $f\equiv 1$, i.e. $M$ is compact, Lemma \ref{lemma:garding2} results from the general estimate about elliptic boundary value problem, c.f. \citep[Ch.~VI, Theorem 4]{Seeley:2011je}.
\end{remark}

\begin{proof}
    The proof proceeds in the same spirit as Lemma \ref{lemma:garding1}. The main difference is that unlike the first order case, the product structure assumption simplifies the calculation here to a large extent. Especially, to take the full advantage of the assumption, an idea due to \citet[pp. 115-117]{Bismut:1991ve} is used to split the estimate into two parts: in the interior and near the boundary. The detail is deferred to Appendix \ref{sec:garding2}.
\end{proof}

	\vspace{2ex}
	\noindent \textbf{Step 2. (Existence of the adjoint operator)} By the  G\"arding type inequality (\ref{eq:garding2}), this step is almost the same as in the first order case.  In fact, the only special feature used here is that the operator $(P_fD)^2$ with domain $\mathscr{B}_+$ is a self-adjoint operator. Just repeat the second step in the proof of Theorem \ref{thm:main1} with product structure assumption. Specifically, for a section $\sigma\in \dom ((P_fD)^2)^*$, one can show:

    \begin{enumerate}
        \item the Fourier expansion coefficients of $\sigma$ have suitable differentiability and satisfy the boundary condition;
        \item the $L^2$ Fourier expansion of $\sigma$ is summable in $\symbf{H}^2$-norm.
    \end{enumerate}
    
    \begin{proof}[Proof of Theorem \ref{thm:sq}]
        After these preparations, one can prove Theorem \ref{thm:sq} easily.
    \end{proof}

    In the remaining part of this subsection, we discuss how to extend Theorem \ref{thm:sq} to the general case. The result is the following theorem.
    \begin{theorem}
    	\label{thm:so}
    	For any $(M,g^{\symup{T}M})$, possibly without product structure near the boundary, the operator
    	\begin{equation*}
            ((P_fD)^2,\mathscr{B}_+):\mathscr{B}_+ \rightarrow \symbf{H}^0_f(M,S_+(\symup{T}M)\otimes E)^G
    	\end{equation*}
    	is a self-adjoint Fredholm operator, where
        \begin{equation}
            \label{eq:sobd}
            \mathscr{B}_+ = \{ \phi\in \symbf{H}^2_f(M,S_+(\symup{T}M)\otimes E)^G|P_{\ge 0,f}(\phi|_{\partial M})= P_{< 0,f}c(-e_n)\big((P_fD \phi)|_{\partial M}\big)=0 \}.
        \end{equation}
    \end{theorem}
    
    \begin{remark}
        \label{rk:so}
        For compact manifold, this is a well-known result which can be established by pseudo-differential calculus, c.f. \citep[Theorem 2.1]{Grubb:1992jd}.
    \end{remark}   
    
    \begin{proof}
    As in Subsection \ref{sub:product_case}, we use the geodesic coordinates $(y,u)$ on $\partial M\times [0,1]$, where $y$ is the coordinate on $\partial M$ and $u$ is the coordinate in the normal direction. To begin with, we fix a metric and a unitary connection on $E$ and that, near the boundary, the metric and connection are constant in the normal direction.
    
    Until now, when mentioning metric on $M$ in this paper, we always refer to metric on the tangent bundle, however the metric on the \uline{cotangent bundle} will be used in this proof. In the geodesic coordinates, the metric of $\symup{T}^*M$ has the following expression,
    \begin{equation*}
    	g^{\symup{T}^*M} = du^2+g^{\symup{T}^* \partial M}_u,
    \end{equation*}
    where $g^{\symup{T}^* \partial M}_{u}$ is a metric on $\symup{T}^*\partial M$ depending on $u$ and $g^{\symup{T}^* \partial M}_{0}$ is actually the induced metric on $\symup{T}^*\partial M$. For convenience, we also introduce a new product metric $g_p^{\symup{T}^* M}$ on $\partial M \times [0,1]$,
    \begin{equation*}
        g_p^{\symup{T}^* M} = du^2+g^{\symup{T}^* \partial M}_{0},
    \end{equation*}
    which is just the pull-back metric induced by the projection from $\partial M\times [0,1]$ to $\partial M$. Since $g_0^{\symup{T}^* \partial M}$ is a $G$-invariant metric on the boundary and the geodesic coordinates is $G$-equivariant, $g_p^{\symup{T}^* M}$ is also a $G$-invariant metric. Then we follow \citet{Grubb:1992jd}, where a family of perturbation metrics $g_{\epsilon}^{\symup{T}^* M}$ is defined. Introduce an auxiliary function $\rho\in \mathscr{C}^\infty([0,1])$, which is a monotonous function taking values in $[0,1]$ and satisfies $ \rho|_{[0,1/3]}=1, \rho|_{[2/3,1]}=0$. For $0< \epsilon \le 1$, on $\partial M\times [0,1]$, set
    \begin{equation}
        g_{\epsilon}^{\symup{T}^* M} =\rho(u/\epsilon) g_p^{\symup{T}^* M} + (1-\rho(u/\epsilon))g^{\symup{T}^* M}. \label{eq:ma}
    \end{equation}
    Obviously, for any $\epsilon> 0$, $g_{\epsilon}^{\symup{T}^* M}$ is of product structure near $\partial M$. Since we are using the geodesic coordinates, and $g^{\symup{T}^* M},g_p^{\symup{T}^* M}$ are $G$-invariant, $g^{\symup{T}^* M}_{\epsilon}$ is $G$-invariant. By (\ref{eq:ma}), $g^{\symup{T}^*M}_{\epsilon}$ will converge to $g^{\symup{T}^*M}$ in $\mathscr{C}^0$-norm when $\epsilon$ approaches to $0$. Since $M$ is non-compact, we should explain more about the meaning of the convergence. Let $U$ be the precompact open subset in (\ref{eq:setr}). By saying that $\{g_{\epsilon}^{\symup{T}^* M}\}$ converge to $g^{\symup{T}^* M}$ in $\mathscr{C}^0$-norm, we mean that $\{g_{\epsilon}^{\symup{T}^* M}|_U\}$ converge to $g^{\symup{T}^* M}|_U$ in $\mathscr{C}^0$-norm on $U$. Due to the $G$-invariance of $\{g_{\epsilon}^{\symup{T}^* M}\}$ and $g^{\symup{T}^* M}$, the convergence is independent of the choice of $U$.
    
    
    Let $D_{\epsilon}$ denote the Dirac operator defined by $g^{\symup{T}^*M}_{\epsilon}$. Obviously, $D_{\epsilon}$ is a $G$-equivariant operator. With $D_{\epsilon}$, one can restate the boundary condition (\ref{eq:sobd}) with an $\epsilon$ parameter. For $0<\epsilon \le 1$, define the following boundary condition,
    \begin{multline}
        \mathfrak{B}_+(\epsilon)=\{\phi\in \symbf{H}^2_f(M,S_+(\symup{T}M)\otimes E)^G| P_{\ge 0,f,\epsilon}(\phi|_{\partial M}) \\
        = P_{< 0,f,\epsilon}c(-e_n)\big((P_fD_{\epsilon} \phi)|_{\partial M}\big)=0\},
    \end{multline}
    where $P_{\ge 0,f,\epsilon}$ and $P_{< 0,f,\epsilon}$ are spectral projection operators defined by $P_fD_{\partial M,\epsilon,+}P_f$. As the notation indicates, $D_{\partial M,\epsilon}$ is the boundary operator of $D_{\epsilon}$.
    
    Note that for any $\epsilon>0$, near the boundary, the operator $D_{\epsilon}$ is equal to $D_p$, which is the Dirac operator defined by the product metric $g_p^{\symup{T}^* M}$. As a result, $\mathfrak{B}_+(\epsilon)$ is independent of $\epsilon$. A key fact about $\mathfrak{B}_+(\epsilon)$ is the following equality due to \citep[Lemma 3.1]{Gilkey:1993dm}, 
    \begin{equation}
        \label{eq:beq}
        \mathscr{B}_+=\mathfrak{B}_+(\epsilon).
    \end{equation}
    
    Having such an interpretation of the boundary condition, we will use the Rellich perturbation theorem, c.f. \citep[Ch. 33, Theorem 5]{Lax:2002ul}, \citep[Theorem X.12]{Reed:1975wp}, to show that $(P_fD)^2$ with domain $\mathscr{B}_+$ is a self-adjoint operator. 
    
    In view of (\ref{eq:beq}), Theorem \ref{thm:sq} implies that $(P_fD_{\epsilon})^2$ with domain $\mathscr{B}_+$ is a self-adjoint Fredholm operator. Since the principle symbol in a local coordinate system of $D^2$ is $g^{ij}\xi^i \xi^j$, where $g^{ij}$ is the coordinate components of $g^{\symup{T}^* M}$, the coefficients of second order term of $D^2-D^2_{\epsilon}$ can be controlled by the $\mathscr{C}^0$-norm of $g^{\symup{T}^* M}-g_{\epsilon}^{\symup{T}^* M}$. Now, for any $\delta>0$, when $\epsilon$ is sufficiently small, by Proposition \ref{prop:tb2}, the following inequality holds for any $s\in \Gamma(M,S(\symup{T}M)\otimes E)^G$,
    \begin{multline}
        \label{eq:sa1}
        \Vert (P_fD)^2(fs) - (P_fD_{\epsilon})^2(fs) \Vert^2_0 \le 2\Vert (D^2-D^2_{\epsilon})(fs) \Vert^2_0 + C_1(\epsilon) \Vert fs \Vert^2_0\\ \le \delta \Vert fs \Vert^2_2 + C_2(\epsilon) \Vert fs \Vert^2_0,
    \end{multline}
    where, as before, we use the trace theorem and the interpolation property of Sobolev spaces to deal with the lower order terms in the second inequality. By the G\"arding type inequality for $((P_fD_{\epsilon})^2,\mathscr{B}_+)$, one has
    \begin{equation}
        \label{eq:sa2}
        \Vert fs \Vert^2_2 \le C_3 \Vert (P_fD_{\epsilon}^2)(fs) \Vert^2_0 + C_4(\epsilon) \Vert fs \Vert^2_0.
    \end{equation}
    The key observation is that by the proof of Lemma \ref{lemma:garding2}, we can choose the constant $C_3$ in (\ref{eq:sa2}) independent of $\epsilon$. Hence, by (\ref{eq:sa1}) and (\ref{eq:sa2}),
    \begin{equation}
        \label{eq:sa3}
        \Vert (P_fD)^2(fs) - (P_fD_{\epsilon})^2(fs) \Vert^2_0 \le \delta C_3 \Vert (P_fD_{\epsilon})^2(fs) \Vert^2_0 + C_5(\epsilon) \Vert fs \Vert^2_0.
    \end{equation}
    Choose $\delta$ and $\epsilon$ correspondingly, such that $\delta C_3<1$. Then by (\ref{eq:sa3}) and the Rellich perturbation theorem, one concludes that $(P_fD)^2$ with domain $\mathscr{B}_+$ is a self-adjoint operator.
    
    Moreover, for any $s\in \Gamma(M,S(\symup{T}M)\otimes E)^G$, by (\ref{eq:sa2}), (\ref{eq:sa3})
    \begin{multline}
        \label{eq:sa4}
        \Vert (P_fD)^2(fs) \Vert^2_0 \ge \frac{1}{2} \Vert (P_fD_{\epsilon})^2(fs) \Vert^2_0 - \Vert (P_fD)^2(fs) -  (P_fD_{\epsilon})^2(fs) \Vert^2_0 \\ \ge (\frac{1}{2}- \delta C_3) \Vert (P_fD_{\epsilon})^2(fs) \Vert^2_0 - C_5(\epsilon) \Vert fs \Vert^2_0 
        \ge (\frac{1}{2C_3}- \delta) \Vert fs \Vert^2_2 - C_6(\epsilon) \Vert fs \Vert^2_0.
    \end{multline}
    By choosing $\delta$ properly such that $\frac{1}{2C_3}- \delta >0$, (\ref{eq:sa4}) turns out to be a G\"arding type inequality, which, as usual, implies the closed range and finite-dimensional kernel properties of $((P_fD)^2,\mathscr{B}_+)$. Combining with the self-adjoint property we just proved, we see that $(P_fD)^2$ with domain $\mathscr{B}_+$ is a self-adjoint Fredholm operator.   
    
    The last step is to remove the additional assumption on $E$. Since different metrics on $E$ lead to equivalent norms, we can choose any invariant metric on $E$ as we wish. As to the connection, one notices that the difference between any two connections on $E$ is a zeroth order operator. Thus when talking about Fredholm property, the operators related to different choices of connections on $E$ coincide up to a compact operator, therefore, having no effect on the conclusion.
    
    \end{proof}
    
    \begin{remark} 
        A few words about the spinor bundles of different metrics. It is pointed out in \citep{LawsonJr:1989ub} that in general there exists no canonical spinor bundle for different metrics. But the key point here is that although metrics (or spinor bundles) change, the boundary condition remains the same, which is enough for proving the Fredholm property of the boundary value problem.
    \end{remark}
    
    \begin{remark}
        \label{rk:ad}
        Needless to say, for the operator $(P_fD)^2$ with domain
        \begin{equation*}
            \mathscr{B}_-=\{ \phi\in \symbf{H}^2_f(M,S_-(\symup{T}M)\otimes E)^G|P_{< 0,f}\big(c(-e_n)\phi|_{\partial M}\big)= P_{\ge 0,f}\big((P_fD \phi)|_{\partial M}\big)=0 \},
        \end{equation*}
        the same proof of Theorem \ref{thm:so} will show its Fredholmness property.
    \end{remark}
    
    Now, we turn to our original goal: solving the Boundary Value Problem \ref{prob:1}, i.e. proving Theorem \ref{thm:main1}, in full generality.
    
    \begin{proof}[Proof of Theorem \ref{thm:main1} in the general case]
    By Remark \ref{rk:gc}, in the general case, one still has the G\"arding type inequality for $(P_fD_+,P_{\ge 0,f})$. As we have used for several times, finite-dimensional kernel and closed range property of $(P_fD_+,P_{\ge 0,f})$ follows. Furthermore, Remark \ref{rk:ad} implies that the dimension of the cokernel of $((P_fD)^2,\mathscr{B}_-)$ is finite. In fact, taking the domain of $(P_fD_+,P_{\ge 0,f})$ and $((P_fD)^2,\mathscr{B}_-)$ into consideration, one has the following inclusion relation about the range of two operators,
    \begin{equation*}
    	\ran ((P_fD)^2,\mathscr{B}_-) \subseteq \ran (P_fD_+,P_{\ge 0,f}),
    \end{equation*}
    from which the finite-dimensional cokernel property of $(P_fD_+,P_{\ge 0,f})$ follows.
    \end{proof}
    
    Recall that $g^{\symup{T}^*M}_{\epsilon}$ is the product metric that approximates the original metric, and $D_{\epsilon}$ is the corresponding Dirac operator. By inspecting the principal symbol of the operator, one knows that the limit operator of $P_fD_{\epsilon}$ when $\epsilon$ approximates $0$, coincides with $P_fD$ up to a zeroth order operator. Therefore, one has the following useful corollary.
    \begin{corollary}
        \label{cor:prod}
        The index of $(P_fD_+,P_{\ge,f})$ is equal to the index of $(P_fD_{\epsilon,+},P_{\ge,f,\epsilon})$ when $\epsilon$ is small enough.
    \end{corollary}
    The above corollary implies that one can always deform the general metric to the product one with the index fixed.
    
    Theorem \ref{thm:main1} justifies our definition of the APS type index on non-compact manifolds. To be more complete, let us discuss the dependence of the index on the choice of cut-off functions to some extent. For two cut-off functions $f_1$ and $f_2$, there exists a natural bijection $T$, between $\mathbf{H}^0_{f_1}(\partial M,(S(\symup{T}M)\otimes E)|_{\partial M})^G$ and $\mathbf{H}^0_{f_2}(\partial M,(S(\symup{T}M)\otimes E)|_{\partial M})^G$, which is induced by the mapping $f_1s\mapsto f_2s$, $s\in \Gamma(\partial M,(S(\symup{T}M)\otimes E)|_{\partial M})^G$. When 
$G$ is unimodular and the cut-off functions are normalized, Proposition \ref{prop:chi} asserts that $T$ is an isometry, which implies that the APS type index is independent of the choice of normalized cut-off functions. For the general group and cut-off functions, however, this is not the case. Such defect can be remedied by noticing the following fact. For a family of cut-off functions $\{f_t,0\le t\le 1\}$, by using $T_t$, the bijection between $\mathbf{H}^0_{f_t}(\partial M,(S(\symup{T}M)\otimes E)|_{\partial M})^G$ and $\mathbf{H}^0_{f_0}(\partial M,(S(\symup{T}M)\otimes E)|_{\partial M})^G$, $\{T_tP_{f_t}D_{\partial M}T^{-1}_t,0\le t\le 1\}$ will be a continous family of closed operators with respect to $t$. Due to this observation, using a result of Kato, c.f. \citep[Ch. IV, \S3.5]{Kato:1991wl}, one can conclude that the eigenvalues of $P_{f_t}D_{\partial M}$ vary continuously with respect to $t$. Following \citep{Atiyah:1976vh}, the spectral flow of $\{P_{f_t}D_{\partial M},0\le t\le 1\}$ can be defined. Using the argument in \citep{Dai:1998ws,Dai:2000vf}, one has the following equality,
    \begin{equation}
        \label{eq:sf}
        \ind (P_{f_1}D_+, P_{\ge 0,f_1}) - \ind (P_{f_0}D_+, P_{\ge 0,f_0}) = \mathrm{sf}\,\{ P_{f_t}D_{\partial M,+}, 0\le t \le 1\},
    \end{equation}
    where $\mathrm{sf}\,\{ P_{f_t}D_{\partial M,+}, 0\le t \le 1\}$ is the spectral flow of $\{P_{f_t}D_{\partial M,+},0\le t\le 1\}$. In this sense, (\ref{eq:sf}) describes the dependence of the APS type index on cut-off functions.


\section{Two indices on non-compact manifolds} 
\label{sec:braverman}

After finding a suitable generalization for the APS type index, we can try to compare two kinds of indices on non-compact manifolds as illustrated in Introduction. In this section, we will restrict to the usual geometric quantization setting. In Subsection \ref{sub:hmind}, we review an index introduced by \citet{Hochs:2015iu}. Subsection \ref{sub:equind} contains the main result of this section, in which we establish an equality between the APS type index and the Hochs-Mathai (HM) type index, which generalizes (\ref{eq:mz2}), more precisely \citep[Theorem 1.5]{Ma:2014ja}, to the non-compact case. In Subsection \ref{sub:ghm}, a generalization of the Hochs-Mathai type index is discussed briefly.

\subsection{The Hochs-Mathai type index} 
\label{sub:hmind}

In a seminal paper \citep{Guillemin:1982ts}, Guillemin and Sternberg worked out the first mathematical rigorous example of a physics principle: $[Q,R]=0$, which has also become widely accepted in mathematics society nowadays. Roughly, it asserts that the geometric quantization procedure commutes with symplectic reduction. To generalize this principle to the case that both the manifold and the group action are non-compact, \citet{Hochs:2015iu} introduced a new index. Let us describe it briefly.

Let $N$ be a non-compact manifold without boundary, on which a locally compact Lie group $G$ acts properly. Notice that we do \uline{not} assume that the $G$-action on $N$ is cocompact. The Lie algebra of $G$ is denoted by $\mathfrak{g}$. Following the usual geometric quantization setting, we assume that: (i) $(N,\omega)$ is a symplectic manifold with a symplectic form $\omega$, and $G$ acts on $N$ symplectically; (ii) there exists a line bundle $L$ over $N$, the so-called prequantum line bundle, and the $G$-action can be lifted up to $L$. About $L$, we further suppose that one can find a $G$-invariant Hermitian metric $h^L$ over it, as well as a $G$-invariant Hermitian connection $\nabla^L$, which satisfies $\frac{\sqrt{-1}}{2\pi}(\nabla^L)^2= \omega$.

Take a $G$-invariant almost complex structure $J$ of $TN$ such that $\omega(\cdot,J\cdot)$ is a $G$-invariant Riemannian metric on $N$ and $J$ preserves $\omega$. As before, the Riemannian metric $\omega(\cdot,J\cdot)$ is denoted by $\langle \cdot, \cdot \rangle$, with respect to which $N$ is assumed to be a complete Riemannian manifold. Using the prequantum line bundle $(L,h^L,\nabla^L)$, one can define the momentum map $\mu:N \rightarrow \mathfrak{g}^*$ by the classical Kostant formula: for $s\in \Gamma(N,L)$ and $X\in \mathfrak{g}$,
\begin{equation}
    \label{eq:kos}
    2\pi \sqrt{-1} (\mu,{X})s = \nabla^L_{X^N}s - L_{X}s,
\end{equation}
where $X^N$ is the vector field generated by $X$ and $L_X$ is the Lie derivative of $X$ on $L$. By (\ref{eq:kos}), one can verify that $\mu$ is an equivariant map and satisfies the following equality,
\begin{equation}
    d(\mu,{X}) = i_{X^N} \omega.
\end{equation}

As in the compact group case, we would like to define the Kirwan vector field by means of the momentum map. The main obstacle here is that, for a non-compact group, $\mathfrak{g}^*$ has no $\mathrm{Ad}^*(G)$-invariant metric. To overcome this difficulty, \citet{Hochs:2015iu} introduced a family of metrics for $\mathfrak{g}^*$ depending on $N$. In fact, they showed the existence of a $G$-invariant metric in the sense of \citep[(13)]{Hochs:2015iu}, $\{(\cdot,\cdot)_n\}_{n\in N}$, on the trivial bundle $N\times \mathfrak{g}^*$. Recall that the $G$-action on $N\times \mathfrak{g}^*$ is defined by
\begin{equation*}
    g\cdot(n,\xi) = (gn, \mathrm{Ad}^*(g)\xi),
\end{equation*}
where $n\in N$, $\xi \in \mathfrak{g}^*$, $g\in G$ and $\mathrm{Ad}^*$ is the coadjoint action on $\mathfrak{g}^*$.

Let $\Phi$ be the metric dual of $\mu$ with respect to $\{(\cdot,\cdot)_n\}_{n\in N}$. As usual, one defines the Kirwan vector field in the following way, for $x\in N$,
\begin{equation*}
    \Phi^N(x) = \big(\dfrac{d}{dt}\Big\vert_{t=0}e^{t \Phi(x)}\big)(x).
\end{equation*}
To give another description of $\Phi^N$, consider the auxiliary function $\widehat{\mathcal{H}}\in \mathscr{C}^{\infty}(N\times N,\mathbb{R})$ defined by
\begin{equation*}
    \widehat{\mathcal{H}}(n,n')= (\mu(n),\mu(n))_{n'},
\end{equation*}
where $n,n'\in N$. Clearly, $\mathcal{H}(n)=\widehat{\mathcal{H}}(n,n)$ is a $G$-invariant function on $N$. Write $d_1 \mathcal{H}$ for the derivative of $\widehat{\mathcal{H}}$ with respect to the first coordinate:
\begin{equation*}
    (d_1 \mathcal{H})_n = \big(d\big(n'\mapsto \widehat{\mathcal{H}}(n',n)\big)\big)_n.
\end{equation*}
Now define vector field $X_1$ to be 
\begin{equation*}
    d_1 \mathcal{H} = \omega(X_1,\cdot) \in \Omega^1(N).
\end{equation*}
In \citep[Lemma 2.3]{Hochs:2015iu}, Hochs and Mathai proved the following equality, which is an analogue of \citep[(1.19)]{Zhang:1998tl},
\begin{equation}
    \label{eq:tz}
    X_1 = 2 \Phi^N.
\end{equation}

Following \citep{Hochs:2015iu}, we make the following compactness assumption about $\Phi^N$.
\begin{assumption}
    $\Phi^N$ is non-vanishing outside a cocompact subset of $N$.
\end{assumption}

\begin{figure}[htbp]
    
    \centering
    \includegraphics[scale=0.4]{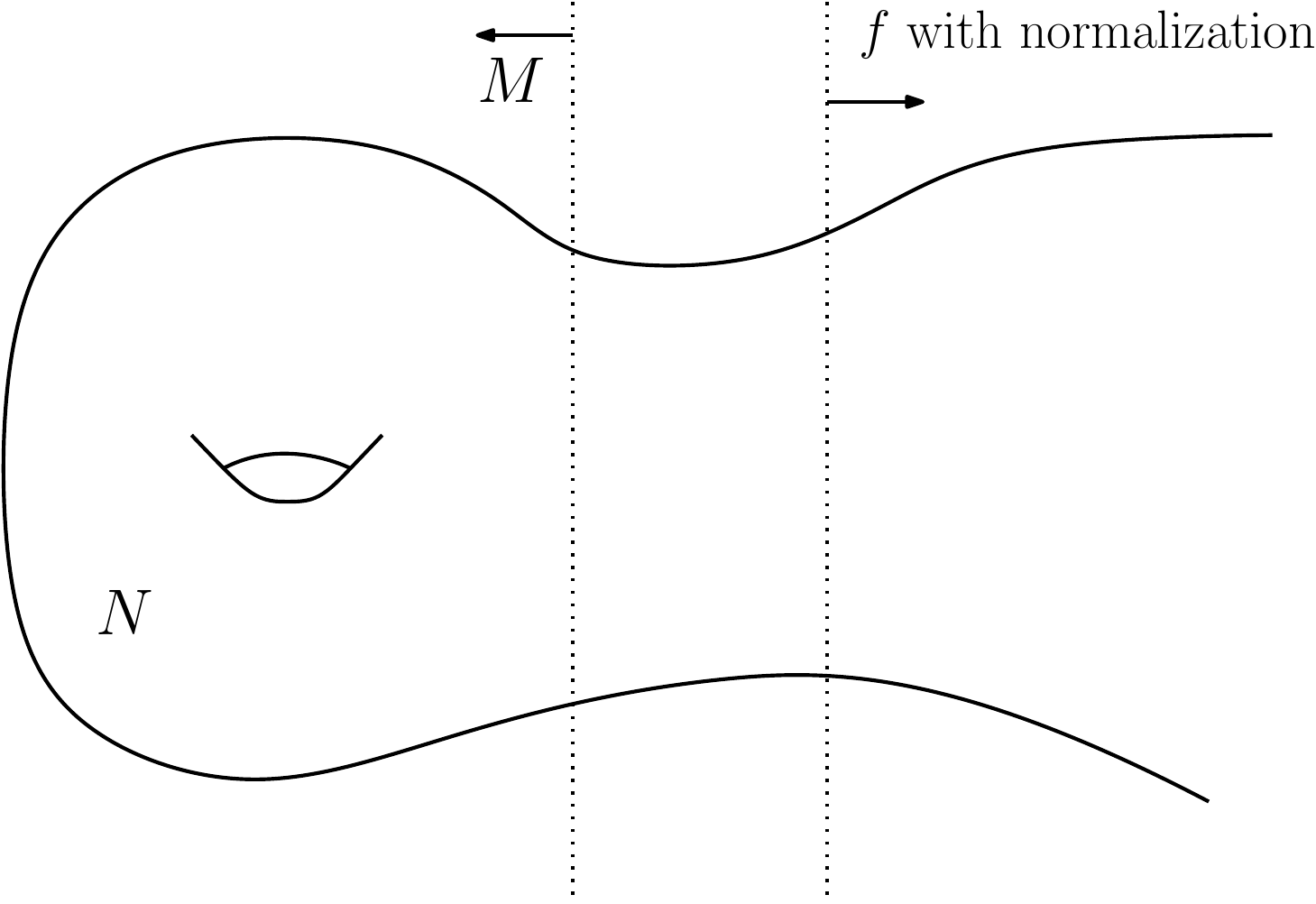}
    \caption{}
    \label{fig:mfd}
\end{figure}

With the above assumption, one can choose a $G$-invariant cocompact submanifold  $M\subseteq N$ that has the same dimension as $N$ and contains all zeroes of $\Phi^N$. Since the $G$-action on $N$ is not cocompact, $M$ has a boundary. Without loss of generality, we can further assume that $\partial M$ is smooth. In fact, due to the above assumption and the $G$-invariance of $\mathcal{H}$, when $a$ is a sufficiently large real number, by Sard's theorem on the regular values, $M$ can be chosen as $\mathcal{H}^{-1}([0,a])$. Moreover, considering the latter usage, one also supposes that for any point $x\in \partial M$, $\Phi^M(x) \neq 0$.\footnote{In view of symbol conformity, $\Phi^N|_M$ will be denoted by $\Phi^M$.} With $M$ fixed, we can specify our choice of cut-off functions used on $N$. To apply the results of \citep{Hochs:2015iu}, we need to combine the cut-off function used in \citep{Hochs:2015iu} and the one used in this paper. Therefore, we make the following assumption on the cut-off function used on $N$. On $M$, which is cocompact, one can use the cut-off function $f$ defined in Section \ref{sec:aps}. Outside an open neighborhood of $M$, whose closure is also cocompact, we require that the cut-off function satisfies the normalization (\ref{eq:xd}). Such a cut-off function on $N$ is still denoted by $f$. The relation between $M$, $N$ and the cut-off function is illustrated in Figure~\ref{fig:mfd}.

As in \citep{Hochs:2015iu}, \citep{Ma:2014ja}, one uses the following deformed Dirac operator on $N$,
\begin{equation*}
    D^{L}_{N,Th} = D^L_{N} + \sqrt{-1}Thc({\Phi}^{N}),
\end{equation*}
where $D^L$ is the Spin$^c$ Dirac operator with coefficient bundle $L$, $T$ is a positive constant and $h\in \mathscr{C}^\infty(N,\mathbb{R})^G$ is an auxiliary positive function which grows rapidly enough near the infinity and satisfies $h|_{M}=1$. For precise conditions on the auxiliary function $h$, we refer to \citep[\S 6.2]{Hochs:2015iu}. When $D^L_{N,Th}$ is restricted to $M$, one recovers the following deformed Dirac operator which has been used in \citep{Tian:1999tq} for manifolds with boundary,
\begin{equation*}        
    D^L_{M,T} = D^L_M + \sqrt{-1}Tc(\Phi^M).
\end{equation*}

To be compatible with the deformed Dirac operator $D^{L}_{N,Th}$, one introduces a new Hilbert space $\symbf{H}^1_{{f},Th,\pm}(N,L)^G$. As a preparation, we single out a useful subspace of smooth sections on $N$. For $s\in \Gamma_{tc}(N,S(\symup{T}N)\otimes L)$, we means that $\supp (s)/G$ is a compact set. Clearly, for such a section, $fs$ will be a section with compact support. Now $\symbf{H}^1_{{f},Th,\pm}(N,L)^G$ is defined to be the completion of $\phi\in{f} \Gamma_{tc}(N,S_{\pm}(\symup{T}N)\otimes L)^G$ with respect to the norm,   
\begin{equation}
    \label{eq:h1n}
    \|\phi\|^2_{h,1} = \|\phi\|^2_0 + \|D^L_{N,Th} \phi\|^2_0.
\end{equation}
Note that when restricted to $M$, which is a manifold with non-empty boundary, the new norm (\ref{eq:h1n}) isn't equivalent to the $\mathbf{H}^1$-norm defined in (\ref{eq:sn}) in general. Nonetheless, the G\"arding type inequality (\ref{eq:gard}) guarantees the equivalence of two norms when restricted to subspaces satisfying the APS boundary condition.

Compared to the usual Witten type deformation for Dirac operators introduced by Tian and Zhang, c.f. \citep{Zhang:1998tl}, \citep{Tian:1999tq}, $D^L_{N,Th}$ involves an auxiliary function $h$. The reason is that to obtain a finite index on a non-compact manifold, 
\citet{Hochs:2015iu} pointed out that one needs to rescale the family of metric $\{(\cdot,\cdot)_n\}_{n\in N}$, which is equivalent to multiplying the Kirwan vector field $\Phi^{N}$ by a function $h$.\footnote{An idea goes back to \citet{Braverman:2002wy}.} Now, one can define the Hochs-Mathai type index as follows.
\begin{definition}
    \label{def:hm}
    The Hochs-Mathai(HM) type index is defined to be the index of the following operator,
    \begin{equation}
    	\label{eq:hm1}
        P_{{f}}D^L_{N,Th,+} = P_{{f}} D^L_{N,+} + \sqrt{-1}Thc({\Phi}^{N}):\, \mathbf{H}^1_{{f},Th,+}(N,{L})^G \rightarrow \mathbf{H}^0_{f}(N,S_-(\symup{T}N)\otimes {L})^G.
    \end{equation}
\end{definition}

To be consistence with convention used in this paper, the definition we use here is a little different from \citep[\S 4.1]{Hochs:2015iu}. But to show the well-definedness of our HM type index, let us recall the definition in \citep{Hochs:2015iu} briefly. For any $G$-invariant section $s$, Hochs and Mathai introduced an operator $\widetilde{D}(fs)=fDs$ and their index is defined to be the index of $\widetilde{D}^L_{N,Th,+}=\widetilde{D}^L_{N,+}+\sqrt{-1}Thc({\Phi}^{N})$ for suitable $h$ and $T$ large enough. The basis to show the Fredholm property of $\widetilde{D}^L_{N,Th,+}$ is the following inequality proved in \citep[Proposition 6.1]{Hochs:2015iu},
\begin{equation}
    \label{eq:hme1}
    \Vert \widetilde{D}^L_{N,Th,+}(fs)\Vert_0^2 \ge C(\Vert\widetilde{D}^L_{N,+}(fs)\Vert^2_0 + T^2\Vert fs \Vert^2_0),
\end{equation}
where $s\in \Gamma_{tc}(N,S_+(\symup{T}N)\otimes {L})^G$ and $C$ is a positive constant. Using the definition of $K_f$ in (\ref{eq:defk}), the difference between $P_fD^L_{N,Th,+}$ and $\widetilde{D}^L_{N,Th,+}$ is
\begin{equation}
    \label{eq:kf}
    fK_fs=(P_fD^L_{N,Th,+}-\widetilde{D}^L_{N,Th,+})(fs).
\end{equation}
By \citep[Proposition 6.6]{Hochs:2015iu}, one can find $h$ such that the following pointwise controls on $K_f$ and $\Phi^N$ hold,
\begin{subequations}
    \begin{align}
        \Vert K_f \Vert &\le C_0 \Vert h \Phi^N \Vert, \label{eq:hme2a} \\
        \Vert h \Phi^N \Vert &\ge 1, \label{eq:hme2b}
    \end{align}
\end{subequations}
where $C_0$ is a positive constant. Moreover, during the proof of \citep[Proposition 6.1, pp.~403]{Hochs:2015iu}, Hochs and Mathai proved the following inequality. With positive constants $C_1,C_2$,
\begin{equation}
    \label{eq:hme3}
    \big((\widetilde{D}^{L,*}_{N,Th,+}\widetilde{D}^L_{N,Th,+} - \widetilde{D}^{L,*}_{N,+}\widetilde{D}^L_{N,+})fs,fs\big) \ge (T^2-C_1 T) (\Vert h \Phi^N \Vert^2fs, fs) - C_2 (fs, fs),
\end{equation}
where the asterisk denotes the adjoint of the corresponding operator. By (\ref{eq:kf}), (\ref{eq:hme2a}), (\ref{eq:hme2b}) and (\ref{eq:hme3}), an inequality similar to (\ref{eq:hme1}) holds for $P_fD^L_{N,Th,+}$,
\begin{equation}
    \label{eq:hme1'}
    \Vert P_f{D}^L_{N,Th,+}(fs)\Vert_0^2 \ge C'(\Vert P_f{D}^L_{N,+}(fs)\Vert^2_0 + T^2\Vert fs \Vert^2_0),
\end{equation}
where $C'$ is another positive constant. By (\ref{eq:hme1'}), using a linear homotopy between $\widetilde{D}^L_{N,Th,+}$ and $P_f{D}^L_{N,Th,+}$ and \citep[Proposition 4.7]{Hochs:2015iu}, one reaches the following proposition.
\begin{proposition}
    \label{prop:hm}
    For suitable $h$ and $T$ large enough, $P_f{D}^L_{N,Th,+}$ has a well-defined index which coincides with the index of $\widetilde{D}^L_{N,Th,+}$ in \emph{\citep{Hochs:2015iu}}.
\end{proposition}

\begin{remark}
    \label{rk:hm}
     If $G$ is unimodular and the cut-off function is normalized, by Proposition \ref{prop:chi} or \citep[(2.26)]{Mathai:2010bt}, one has $P_fD(fs) = \widetilde{D}(fs)$. Thus, in this situation, not only the indices, but also the operators themselves coincide. 
\end{remark}

On the other side, we recall the APS boundary condition for the deformed Dirac operator $P_fD^L_{M,T}$. As in \citep[(1.11)]{Ma:2014ja} and \citep[(1.20)]{Tian:1999tq}, let $D^L_{\partial M, T}$ be the operator induced by $D^L_{M,T}$ on the boundary,
    \begin{equation} 
        \label{eq:pmo}       
        D^L_{\partial M, T} = D^L_{\partial M} - \sqrt{-1}Tc(e_n)c(\Phi^M).
    \end{equation}
    Now the spectral projection operator corresponding to the nonnegative eigenvalues (resp. negative eigenvalues) of $P_f D^L_{\partial M, T}$ will be denoted by $P_{\ge 0,f,T}$ (resp. $P_{<0,f,T}$). With $P_{\ge 0,f,T}$ (resp. $P_{<0,f,T}$), the APS boundary condition for $P_fD^L_{M,T,\pm}$ can be defined as in Definition \ref{def:aps}.
    
    With these preparations, we can state the main result of this section.
    \begin{theorem}
        \label{thm:fin}
        For $T> 0$ large enough, the following equality holds between the Hochs-Mathai type index and the Atiyah-Patodi-Singer type index,
    	\begin{equation}
    		\emph{\text{\texttt{(HM type)}}} \,\ind P_{f}D^L_{N,Th,+} = \ind (P_{f}D^L_{M,T,+},P_{\ge 0,f,T})\,\emph{\text{\texttt{(APS type)}}}.
    	\end{equation}
    \end{theorem}
    
    As in \citep{Ma:2014ja}, we hope this theorem will be useful in studying the geometric quantization problem on non-compact manifolds.
    
    To make the proof of Theorem \ref{thm:fin} more transparent, we rely on a deformation argument. On one side, Corollary \ref{cor:prod} implies that deforming the geometric data near $\partial M$ to product ones will not change the APS type index. On the other hand, since the Kirwan vector field doesn't vanish near $\partial M$, the HM type index won't change during the deformation, either. In fact, in \citep{Braverman:2014uf}, the manifold after the deformation is thought to be cobordant to the original manifold. Hence, to prove Theorem \ref{thm:fin}, one can assume that $N$ has \uline{product structure on $\partial M \times [-1,1]$.}    
    \begin{remark}
        \label{rk:product}
        Admittedly, during the deformation from the general metric to the product one, the symplectic structure near $\partial M$ may be violated. But this is not a problem. In fact, the Hochs-Mathai type index can be defined for Spin$^c$ manifolds. We discuss this general case briefly in Subsection \ref{sub:ghm}.
    \end{remark}
    
    Theorem \ref{thm:fin} will be proved in the next subsection.    

\subsection{An equality about two indices} 
\label{sub:equind}
    
    To begin with, note that due to the $G$-equivariance of $\Phi$, one has the following commutation relation when acting on ${f} \Gamma({M},S(\symup{T}{M})\otimes L)^G$,
    \begin{equation}
        \label{eq:commf}
        [P_f,c(\Phi^M)]=0,
    \end{equation}
    by which, all results in Section \ref{sec:dirac} and Section \ref{sec:product} can be easily generalized to the deformed Dirac operator $D^L_{M,T}$.

    We start with an obvious extension of \citep[Proposition 1.1]{Ma:2014ja}, which corresponds to the $f\equiv 1$ case.
    \begin{lemma}
    	\label{lm:inv}
        If $T$ is sufficiently large, for any section $s \in \Gamma(\partial M, (S(\symup{T}M)\otimes L)|_{\partial M})^G$, the following estimate holds,
    	\begin{equation}
            \label{eq:inv0}
    		\|(P_{f}D^L_{\partial M,T})
            (f s)\|^2_0 \ge \|(P_{f}D^L_{\partial M})(f s)\|^2_0 + C T^2\|f s\|^2_0,
    	\end{equation}
    	for some positive constant $C$.
    \end{lemma}
    \begin{proof}
        By (\ref{eq:pdd}), (\ref{eq:pmo}), (\ref{eq:commf}) and the compatibility of the Clifford connection, one has the following Bochner type formula like \citep[(1.14)]{Ma:2014ja},
    	\begin{multline}
			\label{eq:bo1}		
            (P_fD^{L}_{\partial M,T})^2  = (P_fD^{L}_{\partial M})^2 + \sqrt{-1}T\sum^{n}_{i=1}P_fc(e_i)c(\nabla_{e_i}{\Phi}^{M}) \\ -2\sqrt{-1}TP_f\nabla_{{\Phi}^{M}} + T^2\|{\Phi}^{M}\|^2.
    	\end{multline}
        
        Due to the $G$-equivariance of $\Phi^M$, as in \citep[(1.15)\&(1.16)]{Ma:2014ja}, when acting on $G$-invariant sections, $\nabla_{\Phi^M}$ is a bounded operator. So when acting on $\symbf{H}^0_f(\partial M, (S(\symup{T}M)\otimes L)|_{\partial M})^G$, $P_f\nabla_{{\Phi}^{M}}$ is also a bounded operator. Now using the fact that $\Phi^M$ does not vanish on $\partial M$, there exists $C> 0$ such that 
        \begin{equation}
            \label{eq:fm}
            \Vert \Phi^M \Vert^2 \ge C \quad \text{on } \partial M.
        \end{equation}
        By (\ref{eq:bo1}), (\ref{eq:fm}) and the boundedness of $P_f \nabla_{\Phi^M}$, (\ref{eq:inv0}) is proved.
    \end{proof}
    
    \begin{remark}
        \label{rk:inv}
        Lemma \ref{lm:inv} essentially claims that the boundary operator $P_{f}D^L_{\partial M,T}$ is invertible when $T$ is large enough. Also, we notice that the result of Lemma \ref{lm:inv} does not depend on whether $M$ has product structure near the boundary or not. Proceeding as \citep[Remark 1.4]{Ma:2014ja}, one can conclude that for large $T$, the APS type index of the deformed Dirac operator $(P_fD^L_{M,T,+},P_{\ge 0,f,T})$ is independent of the metric. Moreover, by (\ref{eq:sf}), when $T$ is large enough, the APS type index is also independent of the choices of the cut-off functions.
    \end{remark}
    
    Henceforth, we will assume that $T\gg 0$ is satisfied. To compare two indices defined on $M$ and $N$ respectively, we use a cut-and-paste method. Following \citet{Ma:2014ja}, the manifold $N$ is divided into several parts,
    \begin{equation*}	
    	\begin{gathered}
    		V=(\partial M \times [0,1])\cup (N\backslash M) \subseteq N,\; N= M \cup V,\\
    		Z = \partial M \times [0,1] = V \cap M \subseteq N.
    	\end{gathered}
    \end{equation*}
    Let $\xi$ and $\varphi$ lie in $\mathscr{C}^{\infty}([0,1])$ that satisfy
    \begin{gather*}
    	\xi|_{[0,1/4]}=1,\; 0\le \xi|_{[1/4,3/4]} \le 1,\; \xi|_{[3/4,1]} = 0, \\
    	\varphi = (1-\xi^2)^{1/2}.
    \end{gather*}
    Clearly, $\xi$ and $\varphi$ can be extended to $N$ naturally. Set\footnote{If there is no confusion, the subscript indicating the manifold on which the bundle is defined will be omitted.}
    \begin{equation*}
    	\begin{gathered}
    		H = \symbf{H}^0_{f}(N,S(\symup{T}N)\otimes {L})^G \oplus \symbf{H}^0_{f}(Z,S(\symup{T}Z)\otimes L)^G,\\
    		H' = \symbf{H}^0_{f}(V,S(\symup{T}V)\otimes {L})^G \oplus \symbf{H}^0_{f}(M,S(\symup{T}M)\otimes L)^G.
    	\end{gathered}
    \end{equation*}
    Define a bounded operator $U$ by,
    \begin{equation}
        \label{eq:paste}
        U(s_1,s_2)=(\xi s_1 - \varphi s_2, \varphi s_1 + \xi s_2):H \rightarrow H',
    \end{equation}
    whose adjoint operator $U^* :H'\rightarrow H$ is defined in the usual way. The following commutation relations are easy to check.
    \begin{equation}
    	\label{eq:comm4}
    	[P_{f},U]=0,\quad [P_{f},U^*]=0,
    \end{equation}
    from which one can verify that $U$ indeed maps from $H$ into $H'$, that is, $U$ is well defined. A nice property about $U$ and $U^*$ is that they are isometries, i.e. unitary operators, c.f. \citep{Bunke:1995tu}.
    
    Now we have several manifolds with boundary: $M$, $V$, $Z$. The boundary conditions on them are closely related. If $W$ is $M$, $V$ or $Z$, we define the following Cauchy data space on $W$,
    \begin{equation}
        \label{eq:wbd}
        \symbf{H}^1_{f,Th,+}(W,L;P^W_{\ge 0,f,T})^G=\{ \phi\in \symbf{H}^1_{f,Th,+}(W,L)^G| P^W_{\ge 0,f,T}(\phi|_{\partial W})=0 \},
    \end{equation}
    where $P^W_{\ge 0,f,T}$ is the spectral projection of boundary operator $P_{f}D_{\partial W,Th} = P_{f}D_{\partial W,T}$, considering $h|_M = 1$. For the same reason, if $W=M$ or $Z$, notations on $W$ can be simplified a little,
    \begin{equation*}
        \begin{gathered}
        D^L_{W,Th}=D^L_{W,T},\; \symbf{H}^1_{f,Th,\pm}(W,L;P^W_{\ge 0,f,T})^G = \symbf{H}^1_{f,T,\pm}(W,L;P^W_{\ge 0,f,T})^G.
        \end{gathered}       
    \end{equation*}
    Especially, by the comment after (\ref{eq:h1n}), for $W=M$ or $Z$,  the boundary condition in (\ref{eq:wbd}) is actually the APS boundary condition.
    
    As in \citep{Ma:2014ja}, our scheme is to relate operators defined on $M$ and $N$ by using $U$ and $U^*$. Due to the partition of $N$ we have used, the operators on $Z$ and $V$ enter in the calculation. Nevertheless, their contributions on the index, as we will prove, can be neglected. As a first step, we show that $P_{f}D^{{L}}_{V,Th}$ with boundary (\ref{eq:wbd}) is a Fredholm operator.
    
    By the definition of $U^*$, one can check that the boundary condition in (\ref{eq:wbd}) is preserved by $U^*$, that is, the following map is an isomorphism,
    \begin{multline}
        \label{eq:ustar}
        U^*: \symbf{H}^1_{f,Th,+}(V,{L}; P^{V}_{\ge 0,f, T})^G \oplus \symbf{H}^1_{f,T,+}(M,L;P^M_{\ge 0,f,T})^G\\  \longrightarrow \symbf{H}^1_{f,Th,+}(N,{L})^G \oplus \symbf{H}^1_{f,T,+}(Z,{L}; P^{Z}_{\ge 0,f,T})^G.
    \end{multline}    
As we have argued, the APS type index of $(P_{f}D^L_{Z,T,+},P^Z_{\ge 0,f,T})$ is finite. So is the HM type index of $P_fD^L_{N,Th,+}$. Consequently, by (\ref{eq:paste}), (\ref{eq:ustar}), we assert that the following operator is a Fredholm operator,
    \begin{multline}
        \label{eq:m2o}
        U(P_{f}D^{L}_{N,Th,+}\oplus P_{f}D^L_{Z,T,+})U^*:\symbf{H}^1_{f,T,+}(V,{L};P^{V}_{\ge 0,f,T})^G \oplus \symbf{H}^1_{f,T,+}(M,L;P^M_{\ge 0,f,T})^G\\
    	\longrightarrow \symbf{H}^0_{f}(V,S_-(\symup{T}V)\otimes {L})^G \oplus \symbf{H}^0_{f}(M,S_-(\symup{T}M)\otimes L)^G.
    \end{multline} 
    
    Comparing $U(P_{f}D^{L}_{N,Th,+}\oplus P_{f}D^L_{Z,T,+})U^*$ with $P_{f}D^{{L}}_{V,Th,+} \oplus P_{f}D^L_{M,T,+}$, one can expect that two operators have the same index. That is exactly the case. By (\ref{eq:comm4}) and the definition of $\xi$ and $\varphi$, the difference between the two operators is a zeroth order operator, which must be a compact operator by Rellich's lemma \ref{prop:rellich}. Thereupon, the following operator,
    \begin{equation}
        \label{eq:m2d}
        (P_{f}D^{L}_{V,Th,+},P^{V}_{\ge 0,f,T}): \symbf{H}^1_{f,Th,+}(V;P^{V}_{\ge 0,f,T})^G \rightarrow \symbf{H}^0_{f}(V,S_-(\symup{T}V)\otimes {L})^G,
    \end{equation}
    is a Fredholm operator. Moreover, one has the following index equality,
    \begin{multline}
    	\label{eq:ind}
    	\ind P_{f}D^{L}_{N,Th,+} + \ind (P_{f}D^L_{Z,T,+},P^Z_{\ge 0,f,T}) \\
    	= \ind (P_{f}D^{L}_{V,Th,+},P^{V}_{\ge 0,f,T}) + \ind( P_{f}D^L_{M,T,+},P^M_{\ge 0,f,T}),
    \end{multline}
    because the Fredholm index is insensitive to compact operators.
    
    By virtue of (\ref{eq:ind}), to prove Theorem \ref{thm:fin}, one should show that the indices contributed by $Z$ and $V$ are zero. Since $Z$ has product structure and the inward normal vectors have opposite directions at two boundary components of $Z$, the boundary operators of $P_fD^L_{Z,T,+}$ at two boundary components coincide up to a sign, i.e.,
    \begin{equation}
        \label{eq:bdz}
        P_fD^L_{\partial M\times \{0\},T,+} = -P_fD^L_{\partial M\times \{1\},T,+}.
    \end{equation}
As in (\ref{eq:defp}), one can write down the $L^2$ expansions of sections lying in the kernel of $(P_{f}D^L_{Z,T,+},P^Z_{\ge 0,f,T})$. By (\ref{eq:bdz}), if $\phi\in \ker (P_{f}D^L_{Z,T,+},P^Z_{\ge 0,f,T})$, the boundary condition requires that $\phi$ has the form of $\sum_{\lambda} \mathbb{C}\cdot \eta_{\lambda}$, where $\eta_{\lambda}\in \ker P_fD^L_{\partial M\times \{0\},T,+}$. As a result, Lemma \ref{lm:inv} implies that the kernel of $(P_{f}D^L_{Z,T,+},P^Z_{\ge 0,f,T})$ must be null, which means
    \begin{equation}
        \label{eq:zin}
        \ind (P_{f}D^L_{Z,T,+},P^Z_{\ge 0,f,T})=0.
    \end{equation}
    
    The proof of Theorem \ref{thm:fin} will be complete by using the following lemma about the index of $(P_{f}D^{{L}}_{V,Th,+},P^{V}_{\ge 0,f,T})$.
    \begin{lemma}
    	\label{lm:inv2}
    	For $T$ large enough, the index of operator $(P_{f}D^{{L}}_{V,Th,+},P^{V}_{\ge 0,f,T})$ vanishes.
    \end{lemma}
    
    \begin{proof}
        Like \citep[Lemma 1.6]{Ma:2014ja}, one uses the partition of unity technique of \citet{Bismut:1991ve}, estimating $(P_fD^{L}_{V,Th})^2$ on the following two sets separately,
        \begin{equation*}
            U_1 = N\backslash M,\;U_2 = \partial M \times (-1,1],\;V = U_1\cup U_2.
        \end{equation*}
        
        On the open set $U_1$, by (\ref{eq:hme1'}), for any $s\in \Gamma_{tc}(U_1,S(\symup{T}U_1)\otimes L)^G$ and $T$ large enough,
        \begin{equation}
            \label{eq:u1}
            \Vert P_fD^L_{V,Th}(fs) \Vert^2_0 \ge C_0 \Vert P_fD^L_{V}(fs) \Vert^2_0 + C_1 T^2 \Vert fs \Vert^2_0,
        \end{equation}
         where $C_0,C_1$ are positive constants.
         
         On $U_2$, using \citep[(1.52)]{Ma:2014ja}, for any $s \in \Gamma_{tc}(U_2,S(\symup{T}U_2)\otimes {L})^G$ satisfying boundary condition (\ref{eq:m2d}) and $T$ large enough, one has,
         \begin{equation}
             \label{eq:mze1}
             \|D^L_{V,Th}(fs)\|^2_0 \ge C_2 \|D^L_{V}(fs)\|^2_0 + C_3 T^2 \|fs\|^2_0,
         \end{equation}
         where $C_2$, $C_3$ are positive constants. Notice that compared to \citep[(1.52)]{Ma:2014ja}, we have a cut-off function $f$ in (\ref{eq:mze1}). But it makes no difference because $U_2/G$ is precompact and $f$ has compact support on each orbit of $G$-action. Then by Proposition \ref{prop:tb} and (\ref{eq:mze1}),
         \begin{multline}
             \label{eq:u22}
             \|P_{f}D^L_{V,Th}(fs)\|^2_0 
             \ge \frac{1}{2} \Vert D^L_{V,Th}(fs) \Vert^2_0 - C_4 \Vert fs \Vert^2_0
             \ge  \frac{C_2}{2} \|D^L_{V}(fs)\|^2_0 \\+ (\frac{C_3}{2} T^2 -C_4) \|fs\|^2_0 
             \ge  \frac{C_2}{2} \|P_f D^L_{V}(fs)\|^2_0 + (\frac{C_3}{2} T^2 -C_4) \|fs\|^2_0,
         \end{multline}
         where $C_4$ is a positive constant.
         
         Now choose $h_1,h_2$ to be two smooth $G$-invariant functions on $V$ such that $h^2_1$ and $h^2_2$ form a partition of unity associated with the covering $\{U_1,U_2\}$. Then for any $\sigma\in\Gamma(V, S(\symup{T}V)\otimes L)$, one has $P_fh_j \sigma - h_j P_f \sigma =0$, $j=1,2$. With this fact and using (\ref{eq:u1}) and (\ref{eq:u22}), by an almost word-by-word translation of the proof of \citep[(1.53)-(1.55)]{Ma:2014ja}, one has the following inequality for $(P_{f}D^{{L}}_{V,Th,+},P^{V}_{\ge 0,f,T})$. Let $s$ lie in $\Gamma_{tc}(V,S(\symup{T}V)\otimes L)^G$ and satisfy boundary condition (\ref{eq:m2d}), for $T$ large enough,
         \begin{equation}
             \label{eq:gd}
             \|P_{f}D^L_{V,Th}(fs)\|^2_0 \ge C_5 \|P_fD^L_{V}(fs)\|^2_0 + C_6 T^2 \|fs\|^2_0,
         \end{equation}
         where $C_5$, $C_6$ are positive constants.
         
         By (\ref{eq:gd}), the kernel of $(P_{f}D^L_{V,Th,+},P^{V}_{\ge 0,f,T})$ must be null. Together with (\ref{eq:ind}), one has,
    	\begin{equation}
            \label{eq:ei}	
    		\ind P_{f}D^L_{N,Th,+} -\ind (P_{f}D^L_{M,T,+},P^M_{\ge 0,f,T}) = \ind (P_{f}D^L_{V,Th,+},P^{V}_{\ge 0,f,T}) \le 0.
    	\end{equation}
        
        Likewise for the operator $(P_{f}D^L_{V,Th,-},P^{V}_{< 0,f,T})$, one can prove a formula similar to (\ref{eq:gd}), whose kernel, as a result, also vanishes. Clearly, with such boundary condition, an index equality like (\ref{eq:ind}) also holds. Again, the index on $Z$ vanishes, which implies,
    	\begin{equation} 
            \label{eq:ei2}
            \ind P_{f}D^L_{N,Th,-} -\ind (P_{f}D^L_{M,T,-},P^M_{< 0,f,T}) = \ind (P_{f}D^L_{V,Th,-},P^{V}_{< 0,f,T}) \le 0.
    	\end{equation}
    	On the other hand, the following elementary equalities of indices hold trivially,
        \begin{equation}
            \label{eq:ei3}
            \begin{aligned}
                \ind P_{f}D^L_{N,Th,+}& =-\ind P_{f}D^L_{N,Th,-},\\
                \ind (P_{f}D^L_{M,T,+},P^M_{\ge 0,f,T})& =-\ind (P_{f}D^L_{M,T,-},P^M_{< 0,f,T}).
            \end{aligned}
        \end{equation}
        By (\ref{eq:ei}), (\ref{eq:ei2}), (\ref{eq:ei3}), we have the claimed result,
        \begin{equation*}
            \label{eq:m2in}
            \ind (P_{f}D^L_{V,Th,+},P^{V}_{\ge 0,f,T}) = 0.
        \end{equation*}
    \end{proof}
    
    By (\ref{eq:ind}), (\ref{eq:zin}) and Lemma \ref{lm:inv2}, the proof of Theorem \ref{thm:fin} has finished.
    
\subsection{A generalization of the Hochs-Mathai type index} 
\label{sub:ghm}

    Checking the definition of the Hochs-Mathai type index, if one uses a prescribed equivariant map from $M$ to $\mathfrak{g}$ to replace metric dual of the momentum map, one can define a similar index for the manifold equipped with a Clifford bundle, which includes the Spin$^c$ manifold. In fact, this has been investigated by \citet{Braverman:2002wy,Braverman:2014uf}. All the arguments we have used to prove Theorem \ref{thm:fin} can be rewritten for this general situation. Since the Hochs-Mathai-Braverman type index is independent of $T$, one has the following easy corollary of Theorem \ref{thm:fin}, which can also be proved by a spectral flow argument.
    \begin{corollary}
        \label{cor:inv}
        For $T$ large enough, the APS type index of $(P_{f}D_{M,T,+},P_{\ge 0,f,T})$ is independent of $T$.
    \end{corollary}
        
    As in the HM type index case, this APS type index gives an alternate interpretation of the Braverman index in \citep{Braverman:2014uf}.



\appendix
\def\appendixname{}
\renewcommand*\appendixpagename{\normalsize Appendices}
\appendixpage

\section{Proof of Lemma \ref{lemma:reg}} 
\label{sec:simple}
Essentially, arguments of the same spirit used in this appendix have appeared in \citep{Tang:2013dq}. We collect it here for the sake of convenience. At first, we mention a simple variant of Rellich's lemma.
	\begin{lemma}
		\label{prop:rellich}		
        Let $M$ be a manifold with boundary. The natural embedding,
        \begin{equation*}
            i:\symbf{H}^1_f(M,S(\symup{T}M)\otimes E)^G \longrightarrow \symbf{H}^0_f(M,S(\symup{T}M)\otimes E)^G,
        \end{equation*}
        is compact.
	\end{lemma}
	\begin{proof}
		Since $f$ has compact support, the lemma is just a rephrase of usual Rellich's lemma in our language. 
	\end{proof}
	
    From now on, all notations and assumptions are the same as in Subsection \ref{sub:empty}. Especially, $M$ is a manifold without boundary. Assume $\phi\in \symbf{H}^0_f(M,S(\symup{T}M)\otimes E)^G$ is in the cokernel of $P_fD$, that is, for any $s\in \Gamma(M,S(\symup{T}M)\otimes E)^G$,
	\begin{equation}
		\label{eq:lme2}
        ( P_fD (fs), \phi) = 0.
	\end{equation}
	Let $K_f$ be defined by,
	\begin{equation}
        \label{eq:defk}
		K_f = \frac{1}{A^2(x)} \int_G \delta(g) f(gx) c(df)(gx) dg.
	\end{equation}
	By (\ref{eq:proj}), and the $G$-invariance of $s$, one has
	\begin{equation}
        \label{eq:proje}
		P_fD(fs) = f(D+K_f)s.
	\end{equation}
	 
     By the definition of $\symbf{H}^0_f(M,S(\symup{T}M)\otimes E)^G$, there exists a sequence $\{s_n\}\subseteq \Gamma(M,S(\symup{T}M)\allowbreak\otimes E)^G$ such that $\lim_{n\rightarrow \infty}\Vert fs_n - \phi \Vert_0=0$. Therefore, using the notation in (\ref{eq:setr}), one has $\lim_{n\rightarrow \infty}\Vert s_n|_U - \phi|_U \Vert_{U,0}=0$. Since $G(U)=M$, one can extend $\phi|_U$ in a $G$-invariant way to a locally integrable section $s_0\in \symbf{H}^0_{loc}(M,S(\symup{T}M)\otimes E)^G$ and $fs_0 = \phi$. Also, introduce a $G$-invariant function, $\bar{f}(x)=\int_G f^2(gx) \diff g$. By the definition of $f$, $\bar{f}(x)>0$ for any point $x\in M$. Using (\ref{eq:proje}) and $s_0$, (\ref{eq:lme2}) is equivalent to
	\begin{equation}
		\label{eq:objs0}
         \big((D+K_f)s, f \phi\big) = \big((D+K_f)s, f^2 s_0\big) = 0.
	\end{equation}
    
     Now we prove the following regularity result, which implies the smoothness of $\phi$.
    \begin{proposition}
        \label{prop:regv}
        Suppose that for any section $s \in \Gamma(M,S(\symup{T}M)\otimes E)^G$, \emph{(\ref{eq:objs0})} holds. Then $s_0$ must be a smooth $G$-invariant section.
    \end{proposition}
    
    \begin{proof}
        We assume that $\supp(s)$ is contained in a neighborhood $V$ of a certain orbit of the $G$-action. By \citep[Theorem I.2.1]{Audin:2000we}, $V$ is equivariantly diffeomorphic to $V_0 \times_{G_0} G$, where $V_0$ is an open subset of a Euclidean space and $G_0$ is a compact subgroup of $G$. As an example, we show the proposition with the assumption that $G_0$ is trivial. The general case is proved in the same way due to the compactness of $G_0$. Now by the $G$-invariance of measure $\diff v_M$ and the uniqueness of the Haar measure on $G$, one can find a smooth measure $\diff \mu$ on $V_0$ such that one has the following equality for $\bar{x}\in V_0$ to be the image of $x\in V$,
        \begin{multline}
            \label{eq:regt}
            \big((D+K_f)s, f^2 s_0\big) = \int_{V_0} \langle (D+K_f)s, s_0 \rangle(\bar{x}) \diff \mu(\bar{x}) \int_G f^2(gx) \diff g \\
            = \int_{V_0} \langle (D+K_f)s, \bar{f}s_0 \rangle(\bar{x}) \diff \mu(\bar{x}).
        \end{multline}
        By (\ref{eq:regt}) and the elliptic regularity, we know that $\bar{f}s_0$ is smooth. Considering the positivity of $\bar{f}$, the smoothness of $s_0$ is proved.
    \end{proof}
    
    At last, we comment that, using the same method, one can prove the following result which has been used in the proof of Theorem \ref{thm:bd2}. If there exists some $\sigma \in \symbf{H}^0_f(M,S(\symup{T}M)\otimes E)^G$ such that for any $s\in \Gamma(M,S(\symup{T}M)\otimes E)^G$, the equality $(P_fD(fs),\phi)=(fs,\sigma)$ holds, then one has $\phi \in \symbf{H}^1_f(M,S(\symup{T}M)\otimes E)^G$.
    

\section{Proof of Lemma \ref{lemma:garding2}} 
\label{sec:garding2}

Considering the fact that the G\"arding type inequality basically characterizes some equivalent norms on certain subspace of Sobolev spaces, we must be very careful about which norm we are talking about. Hence, to avoid any possible ambiguity, we write down the $\symbf{H}^2$-norm under consideration explicitly,
	\begin{equation}
        \label{eq:h2norm}
		\|s\|^2_2 = \int_M \left\langle \nabla^2 s, \nabla^2 s \right\rangle 
\diff v_M + \|s\|^2_1,
	\end{equation} 
	where $(\nabla^2 s)(X,Y)=\nabla_X \nabla_Y s-\nabla_{\nabla^{\symup{T}M}_X Y} s$ for any vector fields $X$, $Y$. In the following, constants appearing in inequalities are always positive but may vary from line to line.
    
    By Proposition \ref{prop:tb2}, one can show that for any section $s\in \Gamma(M,S(\symup{T}M)\otimes E)^G$,
    \begin{multline}
        \label{eq:e1}
        \|(P_fD)^2(fs)\|^2_0 = \|D^2(fs)-(1-P_f)D^2(fs)+((P_fD)^2-P_fD^2)(fs)\|^2_0 \\ \ge \frac{1}{2}\|D^2(fs)\|^2_0 - C_1\|fs\|^2_1 - C_2\|fs\|^2_0.
    \end{multline}
    Note that by the proof of Proposition \ref{prop:tb}, $(1-P_f)D^2(fs)$ is a first order operator whose coefficients have compact support. That explains the appearance of $\symbf{H}^1$-norm in (\ref{eq:e1}).
	
    Now we use a trick due to \citet[pp.~115-117]{Bismut:1991ve} to split the estimate of $\|D^2(fs)\|_0^2$ into two parts: (i) $\supp (s)$ lies in the interior of $M$; (ii) $\supp (s)$ lies near the boundary of $M$.
    
When $\supp (s)$ lies in the interior of $M$, by standard elliptic estimates, one has
\begin{equation}
    \label{eq:estint}
    \|D^2(fs)\|^2_0 \ge C_1 \Vert fs \Vert^2_2 - C_2 \Vert fs \Vert^2_0.
\end{equation}

For the other part, we suppose $\supp (s) \subset \partial M \times [0,1)$. Recall that we have assumed that $M$ has product structure near the boundary and $(y,u)$ is the trivial coordinate on $\partial M \times [0,1]$, where $u$ is the coordinate on the normal direction. Around point $p\in \partial M$, let $\{e_i\}^{n-1}_{i=1}$ be a set of locally oriented orthonormal frame of $\symup{T} \partial M$ satisfying $(\nabla^{\symup{T} \partial M}e_i)(p)=0 $. Together with $e_n = \partial/ \partial u$, $\{e_i\}^{n}_{i=1}$ is a set of locally oriented orthonormal frame of $\symup{T}M$ around $p$ satisfying $(\nabla^{\symup{T}M}e_i)(p)=0 $. Using such local frame, one can easily show the following relation between the Bochner Laplacians on $M$ and $\partial M$,
\begin{equation}
\label{eq:eqlap}
        \Delta = \Delta_{\partial M} + \nabla_{e_n}^2 = \Delta_{\partial M} + \partial_u^2,
\end{equation}
where $\partial_u$ is the derivative for the normal direction and $\nabla_{e_n} = \partial_u$ by the product structure assumption.

To keep symbols short, we will use $B$ to represent the boundary operator $D_{\partial M}$ in the following. On $\partial M \times [0,1]$, the Dirac operator takes the form: $D=c(e_n)(\partial_u + B)$. As a result, by (\ref{eq:comm2}), one has
\begin{equation}
    \label{eq:expan}
    \Vert D^2(fs) \Vert^2_0 = \Vert (-\partial^2_u + B^2)(fs) \Vert^2_0 
    = \Vert \partial^2_u (fs) \Vert^2_0 + \Vert B^2(fs) \Vert^2_0 -  2\symup{Re}\big(\partial^2_u (fs), B^2(fs)\big).
\end{equation}

We use a subscript to distinguish the submanifolds on which the Sobolev norms and inner product are calculated. By the standard elliptic estimate, one has
\begin{equation}
    \label{eq:e1b}
    \|B^2(fs)\|^2_{0,\partial M \times \{u\}}
    \ge C_3\Vert fs \Vert^2_{2,\partial M \times \{u\}} - C_4 \Vert fs \Vert^2_{0,\partial M \times \{u\}}.
\end{equation}

By the Lichnerowicz formula on the boundary, one has
\begin{multline}
    \label{eq:lic}
    -2\symup{Re}\big(\partial^2_u (fs), B^2(fs)\big) = 2\symup{Re}\big(\partial^2_u (fs), (\Delta_{\partial M}+ \mathcal{O}(1))(fs)\big) \\
    \ge 2\symup{Re}\big(\partial^2_u (fs), \Delta_{\partial M}(fs)\big) - \epsilon_1 \Vert fs \Vert^2_2 - C(\epsilon_1) \Vert fs \Vert^2_0,
\end{multline}
in which the second inequality we use mean value inequality. Using Green's formula,
\begin{equation}
    \label{eq:gre}
    \big(\partial^2_u (fs), \Delta_{\partial M}(fs)\big) = \big(\partial_u (fs), -\Delta_{\partial M} (fs) \big)_{\partial M \times \{0\}} + \big(\partial_u (fs), -\partial_u \Delta_{\partial M}(fs)\big).
\end{equation}
Using the Lichnerowicz formula for $B$ again to substitute $-\Delta_{\partial M}$ by $B^2$ and combining with Proposition \ref{prop:tb2},
\begin{multline}
    \label{eq:b}
    \big(\partial_u (fs), -\Delta_{\partial M} (fs) \big)_{\partial M \times \{0\}} =  \big(\partial_u (fs), (P_fB)^2(fs) \big)_{\partial M \times \{0\}} \\
    +  \big(\partial_u (fs), \mathcal{O}(1)(fs) \big)_{\partial M \times \{0\}}.
\end{multline}

Using the spectral decomposition of $P_fB$, the boundary conditions (\ref{eq:sqbd}), $P_{\ge 0,f}(fs)=0$, implies the following inequalities on the boundary,
\begin{subequations}
    \begin{align}
        &\big(fs, (P_fB)^i(fs) \big)_{\partial M \times \{0\}} \le 0,\text{ for odd $i$}; \label{eq:dbc1} \\
        &\big(\partial_u (fs),(P_fB)\partial_u(fs)\big)_{\partial M \times \{0\}} \le 0. \label{eq:dbc2}
    \end{align}
\end{subequations}
Similarly, using boundary conditions $P_{< 0,f}(c(-e_n)P_fD(fs))=0$, one has
\begin{multline}
    \label{eq:dbc3}
    \big(c(-e_n)P_fD(fs),P_fB(c(-e_n)P_fD(fs))\big)_{\partial M \times \{0\}} \\
    =\big(\partial_u(fs)+ (P_fB)(fs) , (P_fB)\partial_u(fs) + (P_fB)^2(fs)  \big)_{\partial M \times \{0\}} \ge 0; 
\end{multline}
Using (\ref{eq:dbc1}), (\ref{eq:dbc2}) to cancel off some terms in (\ref{eq:dbc3}), one has
\begin{equation}
    \label{eq:dbc4}
    2\symup{Re}\big(\partial_u (fs), (P_fB)^2(fs) \big)_{\partial M \times \{0\}}\ge 0.
\end{equation}

Note the following expression for the $\symbf{H}^2$-norms on $\partial M \times [0,1]$,
\begin{equation}
    \label{eq:h2n}
    \Vert fs \Vert^2_2 = \Vert \partial^2_u(fs) \Vert^2_0 + 2\sum^{n-1}_{i=1} \Vert \partial_u \nabla_{e_i} (fs) \Vert^2_0 + \int^1_0 \Vert fs \Vert^2_{2,\partial M \times \{u\}} du + \Vert fs \Vert^2_1.
\end{equation}
Notice that in (\ref{eq:gre}), one has $\big(\partial_u (fs), -\partial_u \Delta_{\partial M}(fs)\big)=\sum^{n-1}_{i=1}\Vert \partial_u \nabla_{e_i} (fs) \Vert^2_0$, and in (\ref{eq:e1b}), one can choose $C_3<1$. Plug (\ref{eq:e1b}), (\ref{eq:lic}), (\ref{eq:h2n}) into (\ref{eq:expan}) and using (\ref{eq:dbc4}) to cope with the boundary term, one has
\begin{multline}
    \label{eq:estbd}
    \Vert D^2(fs) \Vert^2_0 \ge (C_5-\epsilon_1) \Vert fs \Vert^2_2 + \big(\partial_u (fs), \mathcal{O}(1)(fs) \big)_{\partial M \times \{0\}} - C_6(\epsilon_1) \Vert fs \Vert^2_0 \\
    \ge (C_5-\epsilon_1 - \epsilon_2) \Vert fs \Vert^2_2 - C_7(\epsilon_1,\epsilon_2) \Vert fs \Vert^2_0,
\end{multline}
where the second inequality is due to the trace theorem of Sobolev spaces. Choose $\epsilon_1$ and $\epsilon_2$ small enough, we get the suitable estimate on $\partial M\times [0,1]$.

Finally, as in \citep[pp.~116]{Bismut:1991ve}, we use a partition of unit $\tau^2_1 + \tau^2_2 = 1$ such that $\supp (\tau_1)\subset [0,3/4)$ and $\supp (\tau_2)\subset (1/4,1]$. The following inequalities hold,
\begin{subequations}
    \begin{gather}
        \Vert D(fs) \Vert^2_0 \ge \frac{1}{2}(\Vert D(f\tau_1 s) \Vert^2_0 + \Vert D(f\tau_2 s)\Vert^2_0) - C_8 \Vert fs \Vert^2_1;\label{eq:pu1} \\
        \Vert f \tau_1 s \Vert^2_2 + \Vert f \tau_2 s \Vert^2_2 \ge \frac{1}{2} \Vert fs \Vert^2_2 - C_9 \Vert fs \Vert^2_1. \label{eq:pu2}
    \end{gather}
\end{subequations}
By (\ref{eq:e1}), (\ref{eq:estint}), (\ref{eq:estbd}) and (\ref{eq:pu1}), (\ref{eq:pu2}), we obtain the estimate (\ref{eq:garding2}) by using the interpolation property of Sobolev spaces.

\vspace{1ex}
\noindent \textit{Acknowledgments.} {The author would like to thank Professor Weiping Zhang for helpful suggestions. And we are indebted to the referees of this paper for their constructive comments.}

\normalem

\bibliographystyle{elsarticle-num-names-sorted}

\bibliography{main}

\end{document}